\documentclass[letterpaper,11 pt]{article}

\usepackage{subfig}
\usepackage{algorithmic}
\usepackage{pifont}
\newcommand{\cmark}{\ding{51}}%
\newcommand{\xmark}{\ding{55}}%

\usepackage{tabularx}

\usepackage{boxedminipage}
\usepackage[justification=centering]{caption}
\def\bko{{\rm 1\kern-.17em l}}
\usepackage{wrapfig}
\usepackage{lipsum}\usepackage{subfig}
\usepackage{algorithmic}
\usepackage{boxedminipage}
\usepackage[justification=centering]{caption}
\def\bko{{\rm 1\kern-.17em l}}
\usepackage{wrapfig}
\usepackage{lipsum}
\usepackage{amsfonts,amsmath,graphicx,fullpage,bbm,amsthm,graphics}
\usepackage{amssymb,multirow,verbatim}
\usepackage{acronym,wrapfig,plain,mathrsfs,enumerate,relsize,color}
\newtheorem{theorem}{Theorem}

\newtheorem{assumption}{Assumption}

\newtheorem{algorithm}{Algorithm}

\newtheorem{lemma}{Lemma}

\newtheorem{proposition}{Proposition}
\newtheorem{remark}{Remark}

\usepackage{algorithm}
\usepackage{subfig}
\usepackage{algorithmic}
\usepackage{boxedminipage}
\usepackage[justification=centering]{caption}
\def\bko{{\rm 1\kern-.17em l}}
\usepackage{wrapfig}
\usepackage{lipsum}

\def\be{\begin{enumerate}}
\def\ee{\end{enumerate}}

\def\argmin{\mathop{\rm argmin}}

 \newcommand{\remove}[1]{}
\newcommand{\EXP}[1]{\mathsf{E}\!\left[#1\right] }

\allowdisplaybreaks

\def\g{\gamma}

\def\argmin{\mathop{\rm argmin}}

\begin{document}
	\title{An iterative regularized mirror descent method for ill-posed nondifferentiable stochastic optimization
	}
	
	\author{Mostafa Amini\thanks{School of Industrial Engineering \& Management, Oklahoma State University, Stillwater, OK 74074, USA,  \texttt{moamini@okstate.edu};}, {Farzad Yousefian}\thanks{School of Industrial Engineering \& Management, Oklahoma State University, Stillwater, OK 74074, USA,
			\texttt{farzad.yousefian@okstate.edu};}
	} 
	
	\setlength{\textheight}{9in} \setlength{\topmargin}{0in}\setlength{\headheight}{0in}\setlength{\headsep}{0in}
	\setlength{\textwidth}{6.5in} \setlength{\oddsidemargin}{0in}\setlength{\marginparsep}{0in}

	\maketitle
	\thispagestyle{empty}

\begin{abstract} 
A wide range of applications arising in machine learning and signal processing can be cast as convex optimization problems. These problems are often ill-posed, i.e., the optimal solution lacks a desired property such as uniqueness or sparsity. In the literature, to address ill-posedness, a bilevel optimization problem is considered where the goal is to find among optimal solutions of the inner level optimization problem, a solution that minimizes a secondary metric, i.e., the outer level objective function. In addressing the resulting bilevel model, the convergence analysis of most existing methods is limited to the case where both inner and outer level objectives are differentiable deterministic functions. While these assumptions may not hold in big data applications, to the best of our knowledge, no solution method equipped with complexity analysis exists to address presence of uncertainty and nondifferentiability in both levels in this class of problems. Motivated by this gap, we develop a first-order method called Iterative Regularized Stochastic Mirror Descent (IR-SMD). We establish the global convergence of the iterate generated by the algorithm to the optimal solution of the bilevel problem in an almost sure and a mean sense. We derive a convergence rate of ${\cal O}\left(1/N^{0.5-\delta}\right)$ for the inner level problem, where $\delta>0$ is an arbitrary small scalar. Numerical experiments for solving two classes of bilevel problems, including a large scale binary text classification application, are presented.
\end{abstract}
	
\section{Introduction}\label{sec:intro}
Consider the following canonical stochastic convex optimization problem

\begin{align}\label{def:firstlevel} \tag{$P_f$}
	\displaystyle \mbox{minimize}& \qquad f(x)\triangleq  \EXP{F(x,\xi)}\\
	\mbox{subject to} &\qquad x \in X, \notag
\end{align}
where $X \subseteq \mathbf{R}^n $ is a nonempty, closed and convex set, $f:X \rightarrow \mathbf{R}$ is a convex function given as an expected value of a stochastic function $F:X\times{\mathbf{R}^{d}} \rightarrow \mathbf{R}$,  $\xi: \Omega \to \mathbf{R}^{d}$ is a random variable, and $(\Omega, {\cal F}, \mathbf{P})$ represents the associated	probability space. In addressing  \eqref{def:firstlevel}, Monte Carlo sampling methods have been very successful in the literature (\cite{Ermoliev69,Ermoliev83}). Of these, the stochastic approximation (SA) method, developed by Robbins and Monro \cite{robbins51sa}, has been applied extensively to solve stochastic optimization and equilibrium problems (\cite{Houyuan08,Farzad3}). Acceleration of SA methods first was introduced by Polyak and Juditsky in '90s \cite{Polyak92} and was carried out by employing averaging techniques. The extension of SA scheme in non-Euclidean spaces was developed by Nemirovski et al. in \cite{Nemir09} and is called stochastic mirror descent (SMD) method. In \cite{Nemir09}, SMD method is applied to solve problem \eqref{def:firstlevel} where function $F(x,\xi)$ is assumed to be nondifferentiable and convex. An optimal convergence rate of  ${\cal O}\left({1}/{\sqrt{N}}\right)$ is derived under averaging. Nedi\'c and Lee \cite{Nedic14} developed SMD methods with an optimal convergence rate under a different set of averaging weights. To address high dimensionality in stochastic optimization, Dang and Lan \cite{Dang15} developed a randomized block-coordinate SMD method in that only a block of the iterate is updated. Optimal non-averaging SMD methods for smooth, nonsmooth, and high dimensional problems with strongly convex objective functions have been also developed (see \cite{Farzad1,Farzad2,Nahid17}).

Often in applications arising from machine learning and signal processing, problem \eqref{def:firstlevel} is ill-posed, i.e., the optimal solution lacks a desired property such as uniqueness or sparsity (see \cite{Tikhonov77} and \cite{Friedlander07} for a detailed review of ill-posed problems and their applications). To address ill-posedness in optimization, a secondary metric is employed that quantifies the desired property. The goal is then to obtain a solution among the optimal solution set of problem \eqref{def:firstlevel} that minimizes the secondary metric. Let function $h: X \rightarrow \mathbf R$ denote the secondary performance measure of interest. Consequently, the following optimization problem is considered 
\begin{align}\label{def:SL} \tag{$P_f^h$}
	\displaystyle \mbox{minimize}& \qquad h(x)\\
	\mbox{subject to} & \qquad x \in \argmin_{y\in X}\EXP{F(y,\xi)} \notag.
\end{align}
Problem \eqref{def:SL} has a bilevel structure and is referred to as the ``selection problem'' (e.g., see \cite{Friedlander07}). The main goal in this paper is to develop a first-order method equipped with complexity analysis for solving problem \eqref{def:SL}.

\begin{remark}
	In some applications, function $h$ can be given in the form of an expectation. As such, throughout the paper, we assume function $h$ is given as $h(x)\triangleq \EXP{H(x,\xi)}$. A motivating example to this case is two-stage stochastic nonlinear programming that will be discussed in the following section (see Lemma \ref{two-stage2bilevel}).
\end{remark}
\begin{remark}
	We note that the term ``bilevel'' has been often used in the literature to refer to a more general formulation, where functions $h$ and $f$ are each characterized in terms of two groups of variables, e.g., $x$ and $y$ (cf. \cite{Dempe12}). However, similar to the terminology used in \cite{Solodov07,Beck14, Sabach17}, throughout this paper, the term ``bilevel'' is used to refer to the specific formulation \eqref{def:SL}. 
\end{remark}

\subsection{Example problems} We discuss two classes of problems that can be formulated using the model \eqref{def:SL}.\\
\noindent (i) \textbf{Ill-posed empirical loss minimization (ELM):}
Given a training set $\{(a_i,b_i)\}_{i=1}^N \subset \mathcal{A}\times \mathcal{B}$ consisting of input objects $a_i$ and their associated output values $b_i$ for datum $i$, the goal in the ELM model lies in learning a function ({e.g.,} a hyperplane in linear regression) in order to classify new observations. The resulting problem is cast as the following convex optimization problem
\begin{align}\label{eqn:problem3}\tag{ELM}
	\displaystyle \mbox{minimize}& \qquad\frac{1}{N}\sum_{i=1}^N \mathcal{L}(a_i^Tx,b_i)\\
	\mbox{subject to} &\qquad x \in X \subseteq \mathbf{R}^n,\notag
\end{align} 
where $\mathcal{L}:\mathbf{R}\times \mathcal{B}\to \mathbf{R}$ is a convex loss function. Depending on the type of the application, a variety of choices for $\mathcal{L}$ have been employed. For instance, in binary classification problems, given an output $b_i \in \{-1,+1\}$, the logistic regression problem is characterized by $\mathcal{L}(z,b_i)=\log(1+\exp(-b_iz))$, while {the hinge loss} is given by $\mathcal{L}(z,b_i)=\max\{0,1-b_iz\}$. Challenges arise when the resulting large-scale problem of the form \eqref{eqn:problem3} is ill-posed. To address ill-posedness, a secondary metric $h(x)$ can be considered. The goal is then to find among optimal solutions of \eqref{eqn:problem3}, one that minimizes $h(x)$ \cite{Mangasarian91,Friedlander07}. For example, to induce sparsity, the \textit{elastic net} regularizer \cite{Zou05} can be considered 
as the secondary metric. Consequently, the following bilevel optimization model is considered \cite{Friedlander07,Sabach17}
\begin{align}\label{prob:ELM-Bilevel}
	\displaystyle \mbox{minimize}& \qquad h(x)\triangleq \|x\|_1 + \mu \|x\|_2^2\\
	\mbox{subject to} &\qquad x \in \arg \min_{y \in X} \EXP{F(y,\xi)}, \nonumber
\end{align}
where, $\xi \in \{\xi_1, \cdots, \xi_N\}$ has a finite support with $Prob(\xi=\xi_i)=1/N$, $F(y,\xi_i)=\mathcal{L}(a_i^Tx,b_i)$, and $\mu>0$ regulates the trade-off between $\ell_1$ and $\ell_2$ norms.

\noindent (ii) \textbf{Two-stage stochastic nonlinear programming:}
In this part, we first consider two-stage stochastic programming (cf. \cite{Birge11} and Ch. 2 of \cite{Shapiro09}) which has a wide area of applications specially in transportation, logistics, finance and power systems \cite{Ankur12,Lan16,Junyi18}. We provide the required preliminaries that help us write a nonlinear two-stage program in the form of a single-stage problem. Then, we show that under some mild assumptions, we can reformulate it as a bilevel problem of the form \eqref{def:SL}. 

Consider the following two-sate stochastic nonlinear programming	
\begin{align}\label{single stage}
	\displaystyle \mbox{minimize}& \qquad   c(z)+\EXP{Q(z,\xi)}  \\
	\mbox{subject to} &\qquad u_\ell(z) \leq 0 , \qquad \hbox{for } \ell=1,\cdots, L, \nonumber \\ 
	&\qquad z \in Z, \nonumber 
\end{align}
for $Z \subseteq \mathbf{R}^n$, functions $c,u_\ell:\mathbf{R}^n \rightarrow  \mathbf{R}$, and a random variable $\xi \in \mathbf{R}^d$ with a finite support $\{\xi_1, \cdots, \xi_N \}$. Here, $Q(z,\xi_i)$ is the optimal value of the following second-stage problem for $i=1,\cdots, N$
\begin{align}\label{second stage}
	\displaystyle \mbox{minimize}& \qquad  q(y_i,\xi_i)  \\
	\mbox{subject to} &\qquad t_j(z) + w_j(y_i,\xi_i) \leq 0, \qquad \hbox{for } j=1,\cdots, J,\nonumber \\ &\qquad y_i \in Y, \nonumber
\end{align}
for the set $Y\subseteq \mathbf{R}^m$, and functions $t_j:\mathbf{R}^n \rightarrow  \mathbf{R}$, $w_j:\mathbf{R}^{m\times d}  \rightarrow  \mathbf{R}$. Note that we assume the random vector here has a finite support. The analysis where it has an infinite support i.e., $\xi \in \{\xi_1, \cdots, \xi_N \}$, is discussed in \cite{Birge11}.

In the following lemma, we show that how we can write the two-stage stochastic programming \eqref{single stage} in a compact form. The proof is provided in Appendix \ref{proof of lemma 2-stage 2}.
\begin{lemma} \label{lemma 2-stage 2}
	Let $Z \subseteq \mathbf{R}^n$ and $Y \subseteq \mathbf{R}^m$ be nonempty, closed and convex sets, functions $c,u_\ell, t_j:\mathbf{R}^n \rightarrow  \mathbf{R}$ be convex over the set $Z$, and function $w_j$ be convex over $Y$ for all $j=1,\cdots,J$. Also, assume $\xi$ is a random variable with finite support $\{\xi_1, \cdots, \xi_N \}$ with $Prob(\xi=\xi_i)=p_i$ for $i=1,\cdots, N$. In addition, suppose $Y_i(z) \triangleq \{y_i\in Y| t_j(z)+w_j(y_i,\xi_i)\leq 0 \hbox{ for } j=1,\cdots,J\}$ is a nonempty set and $q$ is a real-valued convex function over $Y_i(x)$ for $i=1,\cdots, N$. Then, model \eqref{single stage} can be rewritten as follows
	\begin{align}\label{compact two-stage}
		\displaystyle \mbox{minimize}& \qquad  c(z)  + \sum_{i=1}^{N} p_i  q(y_i,\xi_i) \\
		\mbox{subject to} &\qquad u_\ell(z) \leq 0,\qquad \hbox{for } \ell=1, \cdots, L \nonumber \\
		&\qquad t_j(z) + w_j(y_i,\xi_i) \leq  0, \qquad \hbox{for } i=1, \cdots, N,\ j=1,\cdots,J, \nonumber \\ &\qquad \ z \in Z, \ y_i \in Y, \qquad \hbox{for } i=1, \cdots, N. \nonumber
	\end{align}	
\end{lemma}

In the following lemma, we state how we can reformulate the compact model \eqref{compact two-stage} as a bilevel problem of the form \eqref{def:SL}. The proof is given in Appendix \ref{proof of lemma two-stage2bilevel}.
\begin{lemma}\label{two-stage2bilevel}
	Let the random variable $\xi$ have a distribution with a finite support $\{\xi_1,\cdots, \xi_N\}$, and $Prob(\xi=\xi_i)=p_i$ for $i=1, \cdots, N$. Assume $Z \subseteq \mathbf{R}^n$ and $Y \subseteq \mathbf{R}^m$ are nonempty, closed and convex sets. Then, under assumptions given in Lemma \ref{lemma 2-stage 2}, model \eqref{single stage} is equivalent to the following bilevel optimization problem
	\begin{align}\label{bilevel two-stage2}
		\displaystyle \mbox{minimize}& \qquad \EXP{H(x,\xi)}  \\
		\mbox{subject to} &\qquad x \in \argmin_{x \in X } \EXP{F(x,\xi)}, \nonumber 
	\end{align}
	where $x^T \triangleq (z^T,y_1^T,\cdots,y_N^T)$, $X \triangleq Z\times Y^N$, and 
	\begin{align*}
		F(x,\xi_i) &\triangleq \sum_{j=1}^{J} \max \{0, t_j(z)+w_j(y_i,\xi_i) \} + \sum_{\ell=1}^{L} \max \{0,u_\ell(z) \}, \\
		H(x,\xi_i) &\triangleq c(z)+ q(y_i,\xi_i).
	\end{align*}
\end{lemma}
\subsection{Existing methods}  In addressing problem \eqref{def:SL}, challenges may arise due to: (i) the bilevel structure of the problem, (ii) uncertainty, and (iii) nondifferentiability of functions $f$ and $h$. Next, we discuss some of the standard approaches in addressing these challenges for solving \eqref{def:SL} and explain their limitations.
\subsubsection{Sequential regularization (SR)} When problem \eqref{def:firstlevel} is ill-posed, a standard approach is to employ the regularization technique, where a regularized optimization problem of the following form is considered
\begin{align}\label{regularized form} \tag{$P_\lambda$}
	\displaystyle \mbox{minimize}& \qquad f_\lambda(x)\triangleq f(x)+\lambda h(x)\\
	\mbox{subject to} &\qquad x \in X, \notag
\end{align}
where $\lambda>0$ is a (user-specific) regularization parameter and provides a trade-off between the two metrics $f$ and $h$. Examples of this technique include the celebrated Tikhonov regularization \cite{Tikhonov77} where we have $h(x)=\|x\|_2^2$. In signal processing applications, $l_1$ regularization (i.e., $h(x)=\|x\|_1$) has been used extensively to find sparse solutions, i.e. \cite{Boyd07,Beck09}. See \cite{Wright12,Lin14,Kimon16} for a more detailed discussion of the types of regularizers. 
In addressing problem \eqref{def:SL}, one may solve a sequence of the regularized problems \eqref{regularized form} for $\lambda \in\{\lambda_k\} \subset \mathbf R_{++}$ with $\lambda_k \to 0$. This necessitates implementation of a two-loop scheme where in the inner loop, \eqref{regularized form} is solved for a fixed $\lambda$, and in the outer loop, $\lambda$ is updated. As such the sequential regularization scheme is computationally expensive compared to single-loop schemes (see Ch.  12 of \cite{facchinei02finite} for more details). 
\subsubsection{Exact regularization} In addressing ill-posed problems, Mangasarian et al. \cite{Mangasarian79,Mangasarian91} studied ``exact regularization'' of linear and nonlinear programs. A regularization is said to be exact when an optimal solution of \eqref{regularized form}, is also an optimal solution of problem \eqref{def:firstlevel}. Friedlander and Tseng \cite{Friedlander07} showed that the regularization of convex programs is exact when $\lambda$ is below some threshold, and derived error bounds for an inexact regularization. Extensions of this work to variational inequality problems is studied in \cite{Charitha17}. The main drawback of the exact regularization approach is that the threshold on the regularization parameter is not known and is often hard to determine a priori in practice (cf. \cite{Friedlander07}). 
\subsubsection{Iterative regularization (IR)} Another avenue for addressing ill-posedness is the iterative regularization technique. A key difference between SR and IR schemes is that in the latter, the regularization parameter is updated iteratively during the algorithm. As such, IR schemes have a single-loop structure and prove to be computationally more efficient than their SR counterparts. In \cite{Solodov07}, an IR scheme is developed where at the $k$th iteration, an approximate solution to problem \eqref{regularized form} with $\lambda=\lambda_k$ is generated. It is shown that when $\sum_{k=0}^\infty \lambda_k= \infty$, the iterate generated by the proposed method converges to the optimal solution of \eqref{def:SL}. In \cite{Yamada11}, under a similar set of conditions on $\{\lambda_k\}$, a ``hybrid steepest descent method" is developed and the convergence to an optimal solution of the  problem is established. Other papers on addressing problem \eqref{def:SL} include \cite{Elias11,Helou17}. In all the aforementioned papers, the complexity analysis is not addressed. 
Our work in this paper builds on the work in \cite{Farzad3}, where Yousefian et al. considered ill-posed stochastic variational inequality problems where the mapping is merely monotone and possibly non-Lipschitz. In \cite{Farzad3}, an iterative regularized smoothing stochastic approximation scheme, called RSSA, is developed where at each iteration, a noisy observation of the stochastic mapping is used. It is shown that the generated sequence by the RSSA method converges to the least $\ell_2$ norm solution of the VI in an almost sure and a mean sense. Also, a convergence rate of the order $1/\sqrt{k^{1/6-\epsilon}}$ is derived in terms of a suitably defined gap function, where $\epsilon>0$ is an arbitrary small scalar. The main drawback of the RSSA scheme is the degraded convergence rate due to employment of a smoothing scheme. In this paper, this rate is improved to  $1/\sqrt{k^{0.5-\epsilon}}$. Importantly, while in \cite{Farzad3}, the regularizer $h$ is assumed to be $\ell_2$ norm, in this work we allow the function $h$ to be given in the form of an expectation of a nondifferentiable stochastic and strongly convex function. Among the other papers that address the complexity analysis for solving \eqref{def:SL}, \cite{Beck14} and \cite{Sabach17} are described next. 
\subsubsection{Minimal norm gradient method (MNG)} In \cite{Beck14}, a ``minimal norm gradient algorithm'' is developed for solving problem \eqref{def:SL} where $f$ and $h$ are both assumed to be deterministic and differentiable. It is shown that the sequence generated by the algorithm converges to the optimal solution of \eqref{def:SL} (see Theorem 4.1 in \cite{Beck14}). A convergence rate of the order $\frac{1}{\sqrt{k}}$ is derived in terms of $f$ values (see Theorem 4.2 in \cite{Beck14}). The main drawback is that MNG is a two-loop scheme where at each iteration, an optimization problem characterized by the function $h$ needs to be solved. This is computationally expensive in the case that $h$ is complicated by uncertainty. 

\subsubsection{Sequential averaging method (SAM)} In \cite{Sabach17}, a method called BiG-SAM, with an improved convergence rate of the order $\frac{1}{{k}}$ in terms of $f$ values is developed (see Theorem 1 in \cite{Sabach17}). In contrast with \cite{Beck14}, BiG-SAM is a single-loop scheme. An underlying assumption in BiG-SAM is that the function $f$ is of the form $f_1(x)+f_2(x)$, where $f_1$ is continuously differentiable and has Lipschitz gradients and $f_2$ is an extended-valued and possibly nonsmooth function. The differentiability of $f_1$ plays a key role in deriving the sublinear convergence rate in \cite{Sabach17}. Another limitation to \cite{Sabach17} is that both $f$ and $h$ are assumed to be deterministic. In big data applications, $f$ may be stochastic and nondifferentiable. Note that, in solving \eqref{prob:ELM-Bilevel}, the implementation of BIG-SAM becomes challenging due to nondifferentiablity of $h$ and the large sample size $N$.

\begin{table}[t]
	\caption{Comparison of methods in addressing \eqref{def:SL}}
	\centering
	\scalebox{0.85}[0.85]{
		\label{table1}
		\begin{tabular}{c|c|c|c|c|c|c|c|c|c}
			\multirow{2}{*}{Reference}                                             & \multicolumn{3}{c|}{Assump. on $f$}                                             & \multicolumn{3}{c|}{Assump. on $h$}                                              & \multirow{2}{*}{Method}                                                                 & \multirow{2}{*}{Metric} & \multirow{2}{*}{Converg.}   \\ \cline{2-7}
			& con.              & dif.              & form                                  & con.               & dif.              & form                                  &                                                                                         &                         &                             \\ \hline \hline
			\multirow{2}{*}{{\scriptsize Solodov} \cite{Solodov07}}             & \multirow{2}{*}{C} & \multirow{2}{*}{\cmark} & \multirow{2}{*}{$f$}                  & \multirow{2}{*}{C}  & \multirow{2}{*}{\cmark} & \multirow{2}{*}{$h$}                  & \multirow{2}{*}{iter. regu.}                                                            & $f_k-f^*$               & \multirow{2}{*}{asympt.} \\ \cline{9-9}
			&                    &                    &                                       &                     &                    &                                       &                                                                                         & $h_k-h^*$               &                             \\ \hline
			\multirow{2}{*}{{\scriptsize Solodov} \cite{Solodov072}}            & \multirow{2}{*}{C} & \multirow{2}{*}{\xmark} & \multirow{2}{*}{$f$}                  & \multirow{2}{*}{C}  & \multirow{2}{*}{\xmark} & \multirow{2}{*}{$h$}                  & \multirow{2}{*}{iter. regu.}                                                            & $f_k-f^*$               & \multirow{2}{*}{asympt.} \\ \cline{9-9}
			&                    &                    &                                       &                     &                    &                                       &                                                                                         & $h_k-h^*$               &                             \\ \hline
			\multirow{2}{*}{{\scriptsize Beck \& Sabach} \cite{Beck14}}         & \multirow{2}{*}{C} & \multirow{2}{*}{\cmark} & \multirow{2}{*}{$f$}                  & \multirow{2}{*}{SC} & \multirow{2}{*}{\cmark} & \multirow{2}{*}{$h$}                  & \multirow{2}{*}{MNG}                                                                    & $f_k-f^*$               & \scriptsize ${\cal O}(1/\sqrt k)$                 \\ \cline{9-10} 
			&                    &                    &                                       &                     &                    &                                       &                                                                                         & $h_k-h^*$               & asympt.                 \\ \hline
			\multirow{2}{*}{{\scriptsize Sabach \& Shtern} \cite{Sabach17}}     & \multirow{2}{*}{C} & \multirow{2}{*}{\cmark} & \multirow{2}{*}{$f_1+f_2$}            & \multirow{2}{*}{SC} & \multirow{2}{*}{\cmark} & \multirow{2}{*}{$h$}                  & \multirow{2}{*}{SAM}                                                                    & $f_k-f^*$               &\scriptsize  ${\cal O}(1/k)$                       \\ \cline{9-10} 
			&                    &                    &                                       &                     &                    &                                       &                                                                                         & $h_k-h^*$               & asympt.                  \\ \hline
			\multirow{2}{*}{{\scriptsize Garrigos et al.}\cite{Guil18}}        & \multirow{2}{*}{C} & \multirow{2}{*}{\cmark} & \multirow{2}{*}{$f$}                  & \multirow{2}{*}{SC} & \multirow{2}{*}{\cmark} & \multirow{2}{*}{$h$}                  & \multirow{2}{*}{iter. regu}                                                             & $f_k-f^*$               & asymp.                  \\ \cline{9-10} 
			&                    &                    &                                       &                     &                    &                                       &                                                                                         & $h_k-h^*$               & \scriptsize ${\cal O}(1/k)$                       \\ \hline
			\multirow{2}{*}{{\scriptsize Yousefian et al.}\cite{Farzad3}}      & \multirow{2}{*}{C} & \multirow{2}{*}{\cmark} & \multirow{2}{*}{$f$}                  & \multirow{2}{*}{SC} & \multirow{2}{*}{\cmark} & \multirow{2}{*}{$h$}                  & \multirow{2}{*}{iter. regu}                                                             & $f_k-f^*$               & asympt.                  \\ \cline{9-10} 
			&                    &                    &                                       &                     &                    &                                       &                                                                                         & $h_k-h^*$               & \scriptsize ${\cal O}(1/k^{(1/6-\delta)})$        \\ \hline
			\multirow{2}{*}{{\scriptsize Amini \& Yousefian} \cite{Amini18}}     & \multirow{2}{*}{C} & \multirow{2}{*}{\xmark} & \multirow{2}{*}{$\sum_i f_i$}         & \multirow{2}{*}{SC} & \multirow{2}{*}{\xmark} & \multirow{2}{*}{$h$}                  & \multirow{2}{*}{\begin{tabular}[c]{@{}c@{}}incremental\\ iter. regu.\end{tabular}}      & $f_k-f^*$               & \scriptsize ${\cal O}(1/k^{(0.5-\delta)})$        \\ \cline{9-10} 
			&                    &                    &                                       &                     &                    &                                       &                                                                                         & $h_k-h^*$               & asympt.                  \\ \hline
			\multirow{2}{*}{{\scriptsize Kaushik \& Yousefian} \cite{Harshal18}} & \multirow{2}{*}{C} & \multirow{2}{*}{\xmark} & \multirow{2}{*}{high-dim}             & \multirow{2}{*}{SC} & \multirow{2}{*}{\xmark} & \multirow{2}{*}{$h$}                  & \multirow{2}{*}{\begin{tabular}[c]{@{}c@{}}block-coord.\\ iter. regu.\end{tabular}} & $f_k-f^*$               & \scriptsize ${\cal O}(1/k^{(0.5-\delta)})$        \\ \cline{9-10} 
			&                    &                    &                                       &                     &                    &                                       &                                                                                         & $h_k-h^*$               & asympt.                  \\ \hline
			\multirow{2}{*}{\bf This work}                                             & \multirow{2}{*}{C} & \multirow{2}{*}{\xmark} & \multirow{2}{*}{$\EXP{F(\cdot,\xi)}$} & \multirow{2}{*}{SC} & \multirow{2}{*}{\xmark} & \multirow{2}{*}{$\EXP{H(\cdot,\xi)}$} & \multirow{2}{*}{iter. regu}                                                             &\scriptsize $\EXP{f_k}-f^*$               &\scriptsize ${\cal O}(1/k^{(0.5-\delta)})$ a.s.        \\ \cline{9-10} 
			&                    &                    &                                       &                     &                    &                                       &                                                                                         & \scriptsize $\EXP{h_k}-h^*$               & asympt. a.s.                  \\ \hline \hline
		\end{tabular}
	}
		\caption*{{\scriptsize C: Convex, SC: Strongly Convex}}
\end{table}
\subsection{Main contributions}\label{sec:main}
To describe the contributions of our work, we provide Table \ref{table1}. The references \cite{Solodov07,Solodov072} provide no rate statements, while the analysis in \cite{Beck14,Sabach17} relies extensively on differentiability of functions $f$ and $h$. Moreover, in all the references listed in Table \ref{table1} functions in both levels of problem \eqref{def:SL}  are assumed to be deterministic. In this paper,  we allow both functions $f$ and $h$ to be nondifferentiable and complicated by uncertainty. We develop a first-order method called iterative regularized stochastic mirror descent (IR-SMD) (see Algorithm \ref{algorithm:SMD}) where at each iteration, a subgradient of function $f$ is regularized using a regularization parameter and a subgradient of function $h$. This regularization is iterative in the sense that the regularization parameter is updated at each iteration.

Our work is motivated by the idea of iterative regularization which has been studied recently in \cite{Farzad3} and \cite{Garrigos17} for solving variational inequality and optimization problems in ill-posed regimes. Here we apply this technique in solving the optimization problem \eqref{def:SL}. We establish the convergence of the iterate generated by the IR-SMD algorithm to the optimal solution of problem \eqref{def:SL} in both an almost sure and a mean sense (see Theorem \ref{thm conv for xbar}). To perform complexity analysis, we derive a rate of ${\cal O}\left(1/N^{0.5-\delta}\right)$ with respect to function $f$ in the inner level, where $\delta$ is an arbitrary small scalar (see Theorem \ref{thm rate}). To the best of our knowledge, the proposed method in this paper appears to be the first that addresses problem \eqref{def:SL} with rate analysis, when both $f$ and $h$ are nondifferentiable and stochastic.

The remainder of the paper is organized as follows. After presenting the notation, in Section \ref{assum:Prel}, we provide the setup for the prox mapping and outline its main properties. We also discuss the main assumptions on the problem, and show  properties of the sequence of optimal solutions to the regularized problem \eqref{regularized form} (see Proposition \ref{prop:xk_estimate}). In Section \ref{sec:alg}, we present the proposed IR-SMD algorithm and outline the main assumptions of this scheme. In Section \ref{sec:conv}, we prove convergence of the averaging sequence generated by Algorithm \ref{algorithm:SMD} (see Theorem \ref{thm conv for xbar}). The rate of convergence of the proposed method is derived in Section \ref{sec:rate}. Numerical experiments on different nonsmooth problems including a big data text classification application are presented in Section \ref{sec:num}. The paper is ended by some concluding remarks in Section \ref{sec:rem}.

\textbf{Notation:}
The inner product of two vectors $x$ and $y$, both in $ \mathbf R^n$, is shown as $\langle x,y\rangle$. $\EXP{x}$ denotes the expectation of a random variable $x$. We let $\|\cdot\|$ and $\|\cdot\|_*$ denote a general norm and its dual, respectively. The dual norm is defined as $\|x\|_* = \sup \{ \langle x,y \rangle | \ \|y\|\leq 1\}$, for all $x \in \mathbf R^n$. For a convex function $f$ with the domain dom$(f)$, any vector $g_f$ that satisfies $f(x)+\langle g_f, y-x \rangle \leq f(y)$ for all $x,y \in \hbox{dom}(f)$, is called a subgradient of $f$ at $x$. We let $\partial f(x)$ and $\partial  h(x)$ denote the set of all subgradients of functions $f$ and $h$ at $x$. Also, we let $\partial F(x,\xi), \partial H(x,\xi)$ denote the set of all subgradients of functions $F,H$ at $x$ for some $\xi$. Throughout the paper, we let $X^*$ and $x^* \in X^*$ denote the set of optimal solutions and an optimal solution of problem \eqref{def:firstlevel}, respectively. Similarly, we let $X^*_h$ and $x^*_h \in X^*_h$ denote the optimal solution set and an optimal solution of problem \eqref{def:SL}, respectively. We let $x^*_{\lambda}$ denote the optimal solution of problem \eqref{regularized form} and $f^*$ denote the optimal value of problem \eqref{def:firstlevel}.  We use ``a.s.'' to denote almost sure convergence. 
\section{Preliminaries}\label{assum:Prel}
In this section, we present an introduction to the basic concepts that will be employed in our analysis in the subsequent sections.

A distance generating function with respect to norm $\|\cdot\|$ is defined as $\omega:X \rightarrow \mathbf{R}$ when function $\omega$ is smooth and strongly convex with parameter $\mu_\omega>0$, i.e.,
\begin{align}\label{distance1}
	\omega(y) \geq \omega(x) + \langle \nabla \omega(x), y-x \rangle + \frac{\mu_\omega}{2} \|x-y\|^2 \qquad \hbox{for all } x,y \in X.	
\end{align}
Throughout, we assume
\begin{align}\label{distance2}
	\omega(y) \leq \omega(x) + \langle \nabla \omega(x), y-x \rangle + \frac{L_\omega}{2} \|x-y\|^2 \qquad \hbox{for all } x,y \in X,
\end{align}
i.e., $\omega$ has Lipschitz gradients with parameter $L_\omega$. These assumptions have been considered in \cite{Nedic14,Dang15} and hold for example when $\omega(x)=\frac{1}{2}\|x\|^2_2$ for $\mu_\omega=L_\omega=1$.
The Bregman distance $D: X\times X \rightarrow \mathbf{R}$ associated with $\omega$ is defined as follows:
\begin{align}\label{distance3}
	D(x,y) \triangleq \omega (y) - \omega(x)- \langle \nabla \omega(x) , y-x \rangle \qquad \hbox{for all } x,y \in X.	\end{align}
We also define the prox mapping $\mathcal P:X \times \mathbf{R}^n \rightarrow X$ as follows: 
\begin{align}\label{distance4}
	\mathcal{P}_X(x,y) \triangleq \argmin _{z \in X} \{ \langle y,z \rangle + D(x,z)\} \qquad \hbox{for all } x\in X, y \in \mathbf{R}^n.	
\end{align}

In the following lemma, we state some properties of Bregman distance that will be used in this paper. A more comprehensive discussion of these properties can be found in \cite{Nemir09}. 
\begin{lemma}[\bf{Properties of Bregman distance}] \label{D prop}
	Let $D$ be the Bregman distance given by \eqref{distance3}. Then, the following relations hold:
	\begin{itemize}
		\item[(a)] $\frac{\mu_\omega}{2}\|x-y\|^2 \leq D(x,y) \leq \frac{L_\omega}{2}\|x-y\|^2 \qquad$for all $x,y \in X$.
		\item[(b)] $D(x,z)=D(x,y)+D(y,z)+\langle\nabla\omega(y)-\nabla\omega(x),z-y\rangle \qquad$for all $x,y,z \in X$.
		\item[(c)] $\nabla_z D(x,z)=\nabla\omega(z)-\nabla\omega(x) \qquad$for all $x,z \in X$.
	\end{itemize}
\end{lemma}
Next, we state our main assumptions that will be used in the convergence analysis. 
\begin{assumption}[{\bf Problem properties}]\label{assum:properties} 
	Let the following hold:
	\begin{itemize}
		\item[(a)] The set $X \subset \mathbf R^n$ is nonempty, compact, and convex.
		\item[(b)] The function $f(x)$ is subdifferentiable and convex over the set $X$.
		\item[(c)] The function $h(x)$ is  subdifferentiable and strongly convex with parameter $\mu_h>0$ with respect to $\|\cdot\|$; i.e., for all $x,y \in X$ and $g_h(x) \in \partial h(x)$, we have $h(x)+\langle g_h(x), y-x \rangle + \frac{\mu_h}{2}\|x-y\|^2 \leq h(y)$.
		\item[(d)] The stochastic subgradient $g_f(x)$ is such that the following hold almost surely for all $x \in X$
		\begin{align}
			&\EXP{g_F(x,\xi) \mid x}= g_f(x),\label{assumption1:d1}
			\\ &\EXP{\|g_F(x,\xi)\|_*^2} \leq C_F^2,\label{assumption1:d2}
		\end{align}
		where $g_f(x) \in \partial f(x)$, $g_F(x,\xi) \in \partial F(x,\xi)$ and $C_F>0$ is a scalar.
		\item[(e)] The subgradient $g_h(x)$ is such that the following holds almost surely for all $x \in X$
		\begin{align}
			&\EXP{g_H(x,\xi) \mid x}= g_h(x), \label{assumption1:e1}\\
			&\EXP{\|g_H(x,\xi)\|_*^2} \leq C_H^2,\label{assumption1:e2}
		\end{align}
		where $g_h(x) \in \partial h(x)$, $g_H(x,\xi) \in \partial H(x,\xi)$ and $C_H>0$ is a scalar.
	\end{itemize}
\end{assumption}
\begin{remark}\label{assumptionandJensen}
	Note that using Jensen's inequality, from Assumption \ref{assum:properties}(d,e) we have
	\begin{align*}
		&\|g_f(x)\|_*^2=\|\EXP{g_F(x,\xi)}\|_*^2 \leq \EXP{\|g_F(x,\xi)\|_*^2} \leq C_F^2, \\
		&\|g_h(x)\|_*^2=\|\EXP{g_H(x,\xi)}\|_*^2 \leq \EXP{\|g_H(x,\xi)\|_*^2} \leq C_H^2.
	\end{align*} 
\end{remark}
In the following result, we show that under our assumptions, both problems \eqref{def:SL} and \eqref{regularized form} have unique optimal solutions. The proof is provided in Appendix \ref{proof of lem:unique sol for h}.
\begin{lemma}[{\bf Uniqueness of  $\mathbf{x_\lambda^*}$ and $\mathbf{x_h^*}$}]\label{lem:unique sol for h} Let Assumption \ref{assum:properties}(a,b,c) hold. Consider problems \eqref{regularized form} and \eqref{def:SL}. Then,
	\begin{itemize} 
		\item[(a)] Problem \eqref{regularized form} has a unique optimal solution $x_\lambda^*$, for any $\lambda>0$.
		\item[(b)] Problem \eqref{def:SL} has a unique optimal solution $x_h^*$.
	\end{itemize}
\end{lemma}

In the next lemma, we show two inequalities that will be used later in the proof of Proposition \ref{prop:xk_estimate}. The proof is presented in Appendix \ref{proof of results from convexity of f and strong convexity of h}.
\begin{lemma}\label{results from convexity of f and strong convexity of h}
	Let Assumption \ref{assum:properties}(b,c) hold. Suppose $\{\lambda_k\}$ is a sequence of nonnegative scalars. Let $x_{\lambda_k}^*$ be the unique optimal solution of problem ($P_{\lambda_k}$) for $k\geq0$.Then, we have
	\begin{align} 
		&\langle g_f(x_{\lambda_{k-1}}^*) - g_f(x_{\lambda_k}^*) ,x_{\lambda_{k-1}}^*-x_{\lambda_k}^* \rangle \geq 0 \qquad \hbox{for all } k\geq 1, \label{result from convexity of f} \\
		&\langle g_h(x_{\lambda_{k-1}}^*) - g_h(x_{\lambda_k}^*) ,x_{\lambda_{k-1}}^*-x_{\lambda_k}^* \rangle \geq \mu_h \|x_{\lambda_k}^* - x_{\lambda_{k-1}}^*\|^2 \qquad \hbox{for all } k\geq 1. \label{result from strong convexity of h}
	\end{align} 
\end{lemma}

In the next result, considering a sequence of regularized problems \eqref{regularized form} for $\lambda \in \{\lambda_k\}$, we derive an upper bound on the difference between optimal solutions of two regularized problems characterized by a general norm. Importantly, we show that when $\lambda_k$ is decreasing to zero, the trajectory of optimal solutions to the regularized problems, i.e., $\{x^*_{\lambda_k}\}$, converges to the optimal solution  of the problem \eqref{def:SL}, i.e., $x_h^*$. This result is a key to the convergence analysis of our proposed algorithm. The proof is provided in Appendix \ref{proof of prop:xk_estimate}.
\begin{proposition}[{\bf Properties of sequence $\mathbf{\{x_{\lambda_k}^*\}}$}]\label{prop:xk_estimate} 
	Let Assumption \ref{assum:properties} hold. Let  $\{\lambda_k\}$ denote a nonnegative sequence for $k \geq 0$ and $x_{\lambda_k}^*$ be the unique optimal solution to problem ($P_{\lambda_k}$) for $k\geq0$.
	\begin{itemize}
		\item[(a)] Consider problem \eqref{regularized form}. Let $\mu_h$ and $C_H$ be given by Assumption \ref{assum:properties}(c,e). Then, for all $k \geq 1$
		\begin{align}\label{boundxk}\|x_{\lambda_k}^*-x_{\lambda_{k-1}}^*\|\leq \frac{C_H}{\mu_h}\left	|1-\frac{\lambda_{k-1}}{\lambda_k} \right|.\end{align}
		\item[(b)] Consider problem \eqref{def:SL}. When $\lambda_k \rightarrow0$, then the sequence $\{x_{\lambda_k}^*\}$ converges to the unique optimal solution of problem \eqref{def:SL}, i.e., $x_h^*$.
	\end{itemize}
\end{proposition}

\section{Algorithm outline} \label{sec:alg}
In this section, we present the iterative regularized stochastic mirror descent (IR-SMD) method for solving problem \eqref{def:SL}. An outline of this method is presented by Algorithm \ref{algorithm:SMD}. Recall the definition of prox mapping in \eqref{distance4}. In the IR-SMD method, at each iteration, an iterate $x_k$ is updated as follows 
\begin{align}\label{alg}
	x_{k+1} := \mathcal{P}_X(x_k, \gamma_k(g_F(x_k,\xi_k) + \lambda_k g_H(x_k, \tilde \xi_k))) \qquad \hbox{for all } k \geq 0,
\end{align}
where  $\gamma_k>0$ is a proper stepsize, $\lambda_k>0$ is an iterative regularization parameter, $g_F(x_k,\xi_k) \in \partial F(x_k,\xi_k)$, and $g_H(x_k, \tilde \xi_k) \in\partial H(x_k, \tilde \xi_k)$  and $\xi_k,\tilde \xi_k$ are two i.i.d. realizations of random variable $\xi$. A main distinction with the classical SMD method \cite{Nemir09,Nedic14,Dang15} is in terms of the additional regularized term $\lambda_k g_H(x_k, \tilde \xi_k)$ which incorporates the first-order information of the secondary objective function. We consider a weighted average sequence $\{\bar x_k\}$ defined as below: 
\begin{align}\label{weighted alg}
	\bar x_{k+1} := \sum_{t=0}^{k} \eta_{t,k} x_t, \qquad \hbox{where } \eta_{t,k} \triangleq \frac{\gamma_t^r}{\sum_{i=0}^{k} \gamma_i^r},	
\end{align}
in which $r<1$ is a constant. Note that using induction, it can be shown that relation \eqref{weighted alg} is equivalent to \eqref{def:averagingII} in Algorithm \ref{algorithm:SMD} (see e.g., Proposition 3 in \cite{Farzad4}).

An important research question in our work is that how we may update the two parameters $\g_k$ and $\lambda_k$ in order to establish convergence of the averaging sequence $\{\bar x_N\}$, generated by Algorithm \ref{algorithm:SMD}, to the unique optimal solution of problem \eqref{def:SL}. This will be addressed by Theorem \ref{thm conv for xbar}. Another important research question is concerned with the complexity analysis of Algorithm \ref{algorithm:SMD}. This will be addressed in Theorem \ref{thm rate} where under specific update rules for stepsize and regularization parameter, we derive the rate ${\cal O}\left(1/N^{0.5-\delta}\right)$ with respect to $f$ function values.

\begin{algorithm}
	\caption{Iterative Regularized Stochastic Mirror Descent Algorithm (IR-SMD)}
	\label{algorithm:SMD}
	\begin{algorithmic}
		\STATE{\textbf{initialization:} Set a random initial point $x_0\in X$, $\gamma_0>0$ and $\lambda_0>0$ such that $\gamma_0\lambda_0 \leq \frac{L_\omega}{\mu_h}$, a scalar $r<1$, $\bar x_0=x_0 \in \mathbf R^n$, and $S_0=\gamma_0^r$.}
		\FOR{$k=0,1,\cdots,N-1$}
		\STATE{Generate $\xi_k$ and $\tilde \xi_k$ as realizations of random vectors $\xi$.}
		\STATE{Evaluate subgradients $g_F(x_k,\xi_k) \in \partial F(x_k,\xi_k)$ and $g_H(x_k, \tilde \xi_k) \in\partial H(x_k, \tilde \xi_k)$.}
		\STATE{Update $x_k$ using the following relation:}
		\begin{align}\label{mainstep}
		x_{k+1} := \mathcal{P}_X(x_k, \gamma_k(g_F(x_k,\xi_k) + \lambda_k g_H(x_k, \tilde \xi_k))).
		\end{align}
		\STATE{Update $S_k$ and $\bar x_{k}$ using the following recursions:}
		\begin{align}
		&S_{k+1}:=S_k+\gamma_{k+1}^r,\label{def:averagingI} \\
		&\bar x_{k+1}:=\frac{S_k \bar x_k+\g_{k+1}^r x_{k+1}}{S_{k+1}}.\label{def:averagingII} 
		\end{align}
		\STATE{Update the stepsize $\gamma_k$ and regularization parameter $\lambda_k$ (see Theorem \ref{thm conv for xbar} and \ref{thm rate}).}
		\ENDFOR
		\RETURN $\bar x_{N};$
	\end{algorithmic}
\end{algorithm}

We make the following assumption on random variable $\xi$ in the algorithm.
\begin{assumption} [{\bf Random variable $\mathbf{\xi}$}]\label{assum:RVs} For all $k \geq 0$, random variables $\xi_k, \tilde \xi_k \in \mathbf{R}^{d}$ are i.i.d.
\end{assumption}
Throughout, the history of the method is considered as: 
\begin{align}\label{assumption2:f}
	\mathcal{F}_k =\{x_0,\xi_0,\tilde \xi_0, \xi_1, \tilde \xi_1, \cdots, \xi_{k-1}, \tilde \xi_{k-1} \} \qquad \hbox{for all } k\geq 1.
\end{align}

\section{Convergence analysis}
\label{sec:conv}
In this section, our main objective is to establish convergence of the sequence $\{\bar x_N\}$ to the optimal solution of problem \eqref{def:SL}. To this end, we first show convergence properties of the sequence $\{x_k\}$ in Proposition \ref{prop:xk_estimate}. This result will be a key to establish convergence of the averaging sequence which will be presented in Theorem \ref{thm conv for xbar}.

We start with the following result where we characterize the error of the algorithm using a recursive relation in terms of Bregman distance. 
\begin{lemma} [\bf{A recursive upper bound}] \label {recursive bd} Consider problem \eqref{def:SL}. Let the sequence $\{x_k\}$ be generated by Algorithm \ref{algorithm:SMD}. Let Assumption \ref{assum:properties} and \ref{assum:RVs} hold. Also assume $0 < \gamma_k\lambda_k \leq \frac{L_\omega}{\mu_h}$. Then, for all $k\geq 1$ we have
	\begin{align}\label{a recursive upper bound}
		\EXP{D(x_{k+1},x_{\lambda_k}^*)|\mathcal{F}_k}&\leq \left(1-\frac{\mu_h}{2L_\omega}\gamma_k\lambda_k\right)D(x_k,x_{\lambda_{k-1}}^*)+ \frac{2C_H^2L_\omega^3}{\mu_h^3\mu_\omega \gamma_k\lambda_k}\left(\frac{\lambda_{k-1}}{\lambda_k}-1\right)^2\nonumber\\&+\frac{4C_F^2}{\mu_\omega}\gamma_k^2+\frac{4C_H^2}{\mu_\omega}\gamma_k^2\lambda_k^2,
	\end{align}
	where $x_{\lambda_k}^*$ is the unique optimal solution of problem \eqref{regularized form} for $\lambda=\lambda_k$.
\end{lemma}
\begin{proof} 
	Let $k\geq 1 $ be given and $z^*$ be the solution for \eqref{distance4}. From optimality conditions we have
	\begin{align*}
		\langle y + \nabla_{z} D(x,z^*),z-z^* \rangle \geq 0 \qquad \hbox{for all }  x,z\in X, y \in \mathbf{R}^n.
	\end{align*}			
	By Lemma \ref{D prop}(c) we can replace $\nabla_{z^*} D(x,z^*)$ by $\nabla\omega(z^*)-\nabla\omega(x)$. So we obtain
	\begin{align*}
		\langle y + \nabla \omega(z^*)-\nabla \omega(x),z-z^* \rangle \geq 0 \qquad \hbox{for all } x,z\in X, y \in \mathbf{R}^n.
	\end{align*}
	Set $x:= x_k$ and $y:=\gamma_k(g_F(x_k,\xi_k) + \lambda_k g_H(x_k, \tilde \xi_k))$. Note that from \eqref{mainstep} and \eqref{distance4} we have $z^*=x_{k+1}$. We obtain
	\begin{align} \label{proof rate1}
		\langle \gamma_k(g_F(x_k,\xi_k) + \lambda_k g_H(x_k, \tilde \xi_k)) + \nabla \omega(x_{k+1})-\nabla \omega(x_k),z-x_{k+1} \rangle \geq 0 \qquad  \hbox{for all }  z\in X.
	\end{align}
	Consider problem \eqref{regularized form} when $\lambda=\lambda_k$. From optimality conditions we have
	\begin{align}\label{proof rate2}
		\langle g_f(x_{\lambda_k}^*) + \lambda_k g_h(x_{\lambda_k}^*),x-x_{\lambda_k}^* \rangle \geq 0 \qquad  \hbox{for all } x \in X.
	\end{align}
	Letting $z:=x_{\lambda_k}^*$ in \eqref{proof rate1}, $x:=x_{k+1}$ in \eqref{proof rate2} and multiplying it by $\gamma_k$, and adding the resulting inequalities together, we obtain
	\begin{align*}
		\langle \gamma_kg_F(x_k,\xi_k) + \gamma_k\lambda_k g_H(x_k, \tilde \xi_k) &- \gamma_k g_f(x_{\lambda_k}^*) - \gamma_k\lambda_k g_h(x_{\lambda_k}^*)\\&+\nabla \omega(x_{k+1})-\nabla \omega(x_k) ,x_{\lambda_k}^*-x_{k+1} \rangle \geq 0.
	\end{align*} 
	By Lemma \ref{D prop}(b), we know that $\langle \nabla \omega(x_{k+1})-\nabla \omega(x_k) ,x_{\lambda_k}^*-x_{k+1} \rangle= D(x_k,x_{\lambda_k}^*)-D(x_{k+1},x_{\lambda_k}^*)-D(x_k,x_{k+1})$ and also from strong convexity of $\omega$ we have $D(x_k,x_{k+1}) \geq \frac{\mu_\omega}{2} \|x_{k+1}-x_k\|^2$. Combining these relations with the preceding inequality and rearranging the terms we have
	\begin{align*}
		D(x_{k+1},x_{\lambda_k}^*) &\leq D(x_k,x_{\lambda_k}^*) -\frac{\mu_\omega}{2} \|x_{k+1}-x_k\|^2 \\&+\gamma_k\langle g_F(x_k,\xi_k)-g_f(x_{\lambda_k}^*),x_{\lambda_k}^*-x_{k+1} \rangle\\&+ \gamma_k\lambda_k\langle g_H(x_k, \tilde \xi_k)-g_h(x_{\lambda_k}^*),x_{\lambda_k}^*-x_{k+1} \rangle.
	\end{align*}
	By adding and subtracting $x_k$ in the last two terms of the right-hand side 
	we obtain
	\begin{align*}
		D(x_{k+1},x_{\lambda_k}^*) &\leq D(x_k,x_{\lambda_k}^*) -\frac{\mu_\omega}{2} \|x_{k+1}-x_k\|^2 +\gamma_k\langle g_F(x_k,\xi_k)-g_f(x_{\lambda_k}^*),x_{\lambda_k}^*-x_k \rangle\\&+\underbrace{\gamma_k\langle g_F(x_k,\xi_k)-g_f(x_{\lambda_k}^*),x_k-x_{k+1} \rangle}_\text{\hbox{Term1}} \\& +  \gamma_k\lambda_k\langle g_H(x_k, \tilde \xi_k)-g_h(x_{\lambda_k}^*),x_{\lambda_k}^*-x_k \rangle \\&+ \underbrace{\gamma_k\lambda_k\langle g_H(x_k, \tilde \xi_k)-g_h(x_{\lambda_k}^*),x_k-x_{k+1} \rangle}_\text{\hbox{Term2}}.
	\end{align*}
	Note that from Fenchel's inequality, $\langle a,b \rangle \leq \frac{1}{2\alpha}\|a\|^2 +\frac{\alpha}{2}\|b\|_*^2$, for any $a,b \in \mathbf{R}^n$ and $\alpha>0$. Therefore, by applying this relation for Term1 and Term2
	\begin{align*}
		D(x_{k+1},x_{\lambda_k}^*) &\leq D(x_k,x_{\lambda_k}^*) -\frac{\mu_\omega}{2} \|x_{k+1}-x_k\|^2 +
		\gamma_k\langle g_F(x_k,\xi_k)-g_f(x_{\lambda_k}^*),x_{\lambda_k}^*-x_k\rangle\\
		&+ \frac{\gamma_k^2}{\mu_\omega}\|g_F(x_k,\xi_k)-g_f(x_{\lambda_k}^*)\|_*^2 + \frac{\mu_\omega}{4}\|x_{k+1}-x_k\|^2 \\& +
		\gamma_k\lambda_k\langle g_H(x_k, \tilde \xi_k)-g_h(x_{\lambda_k}^*),x_{\lambda_k}^*-x_k \rangle \\&+ \frac{\gamma_k^2\lambda_k^2}{\mu_\omega}\|g_H(x_k, \tilde \xi_k)-g_h(x_{\lambda_k}^*)\|_*^2 + \frac{\mu_\omega}{4}\|x_{k+1}-x_k\|^2.
	\end{align*}
	We obtain
	\begin{align*}
		D(x_{k+1},x_{\lambda_k}^*) &\leq D(x_k,x_{\lambda_k}^*) +\frac{\gamma_k^2}{\mu_\omega}\|g_F(x_k,\xi_k)-g_f(x_{\lambda_k}^*)\|_*^2 \\& +\frac{\gamma_k^2\lambda_k^2}{\mu_\omega}\|g_H(x_k, \tilde \xi_k)-g_h(x_{\lambda_k}^*)\|_*^2+\gamma_k\langle g_F(x_k,\xi_k)-g_f(x_{\lambda_k}^*),x_{\lambda_k}^*-x_k\rangle \\& + \gamma_k\lambda_k\langle g_H(x_k, \tilde \xi_k)-g_h(x_{\lambda_k}^*),x_{\lambda_k}^*-x_k \rangle .
	\end{align*}
	Using the triangular inequality, we have
	\begin{align*}
		D(x_{k+1},x_{\lambda_k}^*) &\leq D(x_k,x_{\lambda_k}^*) +2\frac{\gamma_k^2}{\mu_\omega}\|g_F(x_k,\xi_k)\|_*^2+2\frac{\gamma_k^2}{\mu_\omega}\|g_f(x_{\lambda_k}^*)\|_*^2\\&+2\frac{\gamma_k^2\lambda_k^2}{\mu_\omega}\|g_H(x_k, \tilde \xi_k)\|_*^2+2\frac{\gamma_k^2\lambda_k^2}{\mu_\omega}\|g_h(x_{\lambda_k}^*)\|_*^2\\&+\gamma_k\langle g_F(x_k,\xi_k)-g_f(x_{\lambda_k}^*),x_{\lambda_k}^*-x_k\rangle\\&+ \gamma_k\lambda_k\langle g_H(x_k, \tilde \xi_k)-g_h(x_{\lambda_k}^*),x_{\lambda_k}^*-x_k \rangle .
	\end{align*}
	By Assumption \ref{assum:properties}(d,e) and using the relations in Remark \ref{assumptionandJensen} we have
	\begin{align*}
		D(x_{k+1},x_{\lambda_k}^*) &\leq D(x_k,x_{\lambda_k}^*)+2\frac{\gamma_k^2}{\mu_\omega}C_F^2+2\frac{\gamma_k^2\lambda_k^2}{\mu_\omega}C_H^2+2\frac{\gamma_k^2}{\mu_\omega}\|g_F(x_k,\xi_k)\|_*^2 \\& +2\frac{\gamma_k^2\lambda_k^2}{\mu_\omega}\|g_H(x_k, \tilde \xi_k)\|_*^2+\gamma_k\langle g_F(x_k,\xi_k)-g_f(x_{\lambda_k}^*),x_{\lambda_k}^*-x_k\rangle \\& + \gamma_k\lambda_k\langle g_H(x_k, \tilde \xi_k)-g_h(x_{\lambda_k}^*),x_{\lambda_k}^*-x_k \rangle .
	\end{align*}
	By taking conditional expectation on $\mathcal{F}_k$ and using relations \eqref{assumption1:d2} and \eqref{assumption1:e2}, we have
	\begin{align*}
		\EXP{D(x_{k+1},x_{\lambda_k}^*)|\mathcal{F}_k} &\leq D(x_k,x_{\lambda_k}^*)+4\frac{\gamma_k^2}{\mu_\omega}C_F^2+4\frac{\gamma_k^2\lambda_k^2}{\mu_\omega}C_H^2 \\& +\gamma_k\langle \EXP{g_F(x_k,\xi_k)|\mathcal{F}_k}-g_f(x_{\lambda_k}^*),x_{\lambda_k}^*-x_k\rangle \\& +  \gamma_k\lambda_k \left \langle \EXP{g_H(x_k, \tilde \xi_k)|\mathcal{F}_k}-g_h(x_{\lambda_k}^*),x_{\lambda_k}^*-x_k \right \rangle.
	\end{align*}
	Using relations \eqref{assumption1:d1} and \eqref{assumption1:e1}, we obtain
	\begin{align*}
		\EXP{D(x_{k+1},x_{\lambda_k}^*)|\mathcal{F}_k} &\leq D(x_k,x_{\lambda_k}^*)+4\frac{\gamma_k^2}{\mu_\omega}C_F^2+4\frac{\gamma_k^2\lambda_k^2}{\mu_\omega}C_H^2\\&+\gamma_k\langle g_f(x_k)-g_f(x_{\lambda_k}^*),x_{\lambda_k}^*-x_k\rangle\\&+ \gamma_k\lambda_k\langle g_h(x_k)-g_h(x_{\lambda_k}^*),x_{\lambda_k}^*-x_k \rangle.
	\end{align*}
	Similar to the proof of Lemma \ref{results from convexity of f and strong convexity of h}, by convexity of $f$ and strong convexity of $h$ we know that $\langle g_f(x_k)-g_f(x_{\lambda_k}^*),x_{\lambda_k}^*-x_k\rangle \leq 0$ and $\langle g_h(x_k)-g_h(x_{\lambda_k}^*),x_{\lambda_k}^*-x_k \rangle \leq -\mu_h\|x_k-x_{\lambda_k}^*\|^2 \leq -\frac{\mu_h}{2}\|x_k-x_{\lambda_k}^*\|^2$, so
	\begin{align*}
		\EXP{D(x_{k+1},x_{\lambda_k}^*)|\mathcal{F}_k} \leq D(x_k,x_{\lambda_k}^*)+4\frac{\gamma_k^2}{\mu_\omega}C_F^2+4\frac{\gamma_k^2\lambda_k^2}{\mu_\omega}C_H^2-\frac{\gamma_k\lambda_k\mu_h}{2}\|x_k-x_{\lambda_k}^*\|^2.
	\end{align*}
	From Lemma \ref{D prop}(a), $D(x,y) \leq \frac{L_\omega}{2}\|x-y\|^2$ for all $x,y \in X$ so,
	\begin{align} \label{proofrate3}
		\EXP{D(x_{k+1},x_{\lambda_k}^*)|\mathcal{F}_k} \leq \left(1-\frac{\gamma_k\lambda_k\mu_h}{L_\omega}\right)D(x_k,x_{\lambda_k}^*)+4\frac{\gamma_k^2}{\mu_\omega}C_F^2+4\frac{\gamma_k^2\lambda_k^2}{\mu_\omega}C_H^2.
	\end{align}
	Next we relate $D(x_k,x_{\lambda_k}^*)$ to $D(x_k,x_{\lambda_{k-1}}^*)$. By Lemma \ref{D prop}(b) we have
	\begin{align*}
		D(x_k,x_{\lambda_k}^*)=D(x_k,x_{\lambda_{k-1}}^*)+D(x_{\lambda_{k-1}}^*,x_{\lambda_k}^*)+ \underbrace{\langle \nabla\omega(x_{\lambda_{k-1}}^*)-\nabla\omega(x_k), x_{\lambda_k}^*-x_{\lambda_{k-1}}^* \rangle}_\text{\hbox{Term3}}.
	\end{align*}	
	By multiplying and dividing the term $\sqrt{\frac{\mu_h\mu_\omega \gamma_k\lambda_k}{2L_\omega^3}}$ in Term3 and using Fenchel's inequality we obtain
	\begin{align*}
		D(x_k,x_{\lambda_k}^*) &\leq D(x_k,x_{\lambda_{k-1}}^*)+D(x_{\lambda_{k-1}}^*,x_{\lambda_k}^*)+ \frac{\mu_h\mu_\omega \gamma_k\lambda_k}{4L_\omega^3} \|\nabla\omega(x_{\lambda_{k-1}}^*)-\nabla\omega(x_k)\|_*^2 \\&+\frac{L_\omega^3}{\mu_h\mu_\omega \gamma_k\lambda_k} \|x_{\lambda_k}^*-x_{\lambda_{k-1}}^*\|^2,
	\end{align*}
	where $L_\omega$ is the Lipschitz parameter of $\nabla \omega$ defined in \eqref{distance2}. By definition of Lipschitzian property we know that $\|\nabla \omega(x) -\nabla \omega(y)\|_* \leq L_\omega \|x-y\|$ for all $x,y \in X$. Therefore, from the preceding relation we obtain
	\begin{align*}
		D(x_k,x_{\lambda_k}^*) & \leq D(x_k,x_{\lambda_{k-1}}^*)+D(x_{\lambda_{k-1}}^*,x_{\lambda_k}^*)\\& + \frac{\mu_h\mu_\omega \gamma_k\lambda_k}{4L_\omega} \|x_{\lambda_{k-1}}^*-x_k\|^2 +\frac{L_\omega^3}{\mu_h\mu_\omega \gamma_k\lambda_k} \|x_{\lambda_k}^*-x_{\lambda_{k-1}}^*\|^2.
	\end{align*} 
	From $D(x,y) \leq \frac{L_\omega}{2}\|x-y\|^2$ for all $x,y \in X$, and also using Proposition \ref{prop:xk_estimate}(a) we have
	\begin{align*}
		D(x_k,x_{\lambda_k}^*) &\leq D(x_k,x_{\lambda_{k-1}}^*)+\frac{L_\omega C_H^2}{2\mu_h^2}\left|1-\frac{\lambda_{k-1}}{\lambda_k}\right|^2 \\& + \frac{\gamma_k\lambda_k\mu_h\mu_\omega}{4L_\omega} \|x_{\lambda_{k-1}}^*-x_k\|^2 +\frac{L_\omega^3 C_H^2}{\gamma_k\lambda_k\mu_h^3\mu_\omega} \left|1-\frac{\lambda_{k-1}}{\lambda_k}\right|^2.
	\end{align*}	
	From Lemma \ref{D prop}(a) we have $\|x_{\lambda_{k-1}}^*-x_k\|^2 \leq \frac{2}{\mu_\omega}D(x_k,x_{\lambda_{k-1}}^*)$. Taking this into account and by rearranging the terms in the preceding relation we have
	\begin{align*}
		D(x_k,x_{\lambda_k}^*) \leq \left(1+\frac{\mu_h \gamma_k\lambda_k}{2L_\omega}\right) D(x_k,x_{\lambda_{k-1}}^*)+\frac{L_\omega C_H^2}{2\mu_h^2}\left(1+\frac{2L_\omega^2}{\mu_h\mu_\omega \gamma_k\lambda_k}\right)\left(\frac{\lambda_{k-1}}{\lambda_k}-1\right)^2.
	\end{align*}	
	Replacing the preceding inequality in \eqref{proofrate3} since $\gamma_k\lambda_k \leq \frac{L_\omega}{\mu_h}$, and considering Lipschitzian property of $\omega$ we obtain
	\begin{align*}
		\EXP{D(x_{k+1},x_{\lambda_k}^*)|\mathcal{F}_k} &\leq \left(1-\frac{\mu_h \gamma_k\lambda_k}{L_\omega}\right) \left(1+\frac{\mu_h\gamma_k\lambda_k}{2L_\omega}\right)D(x_k,x_{\lambda_{k-1}}^*) \nonumber\\&+\left(1-\frac{\mu_h\gamma_k\lambda_k}{L_\omega}\right)\frac{L_\omega C_H^2}{2\mu_h^2}\left(1+\frac{2L_\omega^2}{\mu_h\mu_\omega \gamma_k\lambda_k}\right)\left(\frac{\lambda_{k-1}}{\lambda_k}-1\right)^2\\& +4\frac{\gamma_k^2}{\mu_\omega}C_F^2+4\frac{\gamma_k^2\lambda_k^2}{\mu_\omega}C_H^2.
	\end{align*}
	By rearranging the terms we have
	\begin{align*}
		\EXP{D(x_{k+1},x_{\lambda_k}^*)|\mathcal{F}_k} &\leq \underbrace{ \left(1-\frac{\mu_h\gamma_k\lambda_k}{2L_\omega} - \frac{(\mu_h\gamma_k\lambda_k)^2}{2L_\omega^2}\right)D(x_k,x_{\lambda_{k-1}}^*)}_\text{\hbox{Term4}} \nonumber\\&+\underbrace{\left(1-\frac{\mu_h\gamma_k\lambda_k}{L_\omega}\right)\frac{L_\omega C_H^2}{2\mu_h^2}\left(1+\frac{2L_\omega^2}{\mu_h\mu_\omega\gamma_k\lambda_k}\right)\left(\frac{\lambda_{k-1}}{\lambda_k}-1\right)^2}_\text{\hbox{Term5}}\\& +\frac{4C_F^2}{\mu_\omega}\gamma_k^2+\frac{4C_H^2}{\mu_\omega}\gamma_k^2\lambda_k^2.
	\end{align*}
	We can drop the nonpositive term $- \frac{(\mu_h\gamma_k\lambda_k)^2}{2L_\omega^2}$ in Term4. Also note that  we can drop the nonpositive term $-\frac{\mu_h\gamma_k\lambda_k}{L_\omega}$ in Term5. We obtain
	\begin{align*}
		\EXP{D(x_{k+1},x_{\lambda_k}^*)|\mathcal{F}_k} &\leq \left(1-\frac{\mu_h\gamma_k\lambda_k}{2L_\omega} \right)D(x_k,x_{\lambda_{k-1}}^*) \nonumber\\&+\underbrace{\frac{L_\omega C_H^2}{2\mu_h^2}\left(1+\frac{2L_\omega^2}{\mu_h\mu_\omega\gamma_k\lambda_k}\right)\left(\frac{\lambda_{k-1}}{\lambda_k}-1\right)^2}_\text{\hbox{Term6}}+\frac{4C_F^2}{\mu_\omega}\gamma_k^2+\frac{4C_H^2}{\mu_\omega}\gamma_k^2\lambda_k^2.
	\end{align*}
	Note that we have $\mu_\omega \leq L_\omega$. Combining this relation with the assumption that $\gamma_k\lambda_k \leq \frac{L_\omega}{\mu_h}$, we have $ 1 \leq \frac{2L_\omega^2}{\mu_h\mu_\omega\gamma_k\lambda_k}$. From this relation and the preceding relation, we obtain
	\begin{align*}
		\EXP{D(x_{k+1},x_{\lambda_k}^*)|\mathcal{F}_k} &\leq \left(1-\frac{\mu_h}{2L_\omega}\gamma_k\lambda_k\right)D(x_k,x_{\lambda_{k-1}}^*) \\& +\frac{2C_H^2L_\omega^3}{\mu_h^3\mu_\omega \gamma_k\lambda_k}\left(\frac{\lambda_{k-1}}{\lambda_k}-1\right)^2+\frac{4C_F^2}{\mu_\omega}\gamma_k^2+\frac{4C_H^2}{\mu_\omega}\gamma_k^2\lambda_k^2.
	\end{align*}
\end{proof}

The next result will be utilized to establish the convergence of Algorithm \ref{algorithm:SMD}. To this end, we will employ this result in proving convergence of the iterate $x_k$ in Proposition \ref{conv rate}.

\begin{lemma} [\textbf{Lemma 11, pg.\ 50 of~\cite{Polyak87}}] \label {lemma conv lemma}
	Let $\{\nu_k\}$ be a sequence of nonnegative random variables, where $\EXP{\nu_0}<\infty$, and let $\{\alpha_k\}$ and $\{\beta_k\}$ be deterministic scalar sequences such that: 
	\begin{align*}
		&\EXP{\nu_{k+1}|\nu_0,\dots,\nu_k} \leq (1-\alpha_k)\nu_k + \beta_k \qquad \hbox{for all } k \geq 0,\\
		&0 \leq \alpha_k \leq 1, \ \beta_k \geq0, \ \sum_{k=0}^{\infty}\alpha_k=\infty, \ \sum_{k=0}^{\infty}\beta_k <\infty, \ \hbox{and} \ \lim_{k \to \infty} \frac{\beta_k}{\alpha_k}=0.
	\end{align*}
	Then, $\nu_k \rightarrow 0$ almost surely, and $\lim_{k\to \infty}\EXP{\nu_k}=0$.
\end{lemma} 
An extension of Lemma \ref{lemma conv lemma} is proposed in the following result. We will employ this result in Proposition \ref{conv rate}(c) to derive a rate statement. This lemma is proved in Appendix \ref{proof of lemma general conv lemma}. 
\begin{lemma} \label {lemma general conv lemma} 
	Let $\{\nu_k\}$ be a sequence of nonnegative random variables, where \\ $\EXP{\nu_0}<\infty$, and let $\{\alpha_k\}$ and $\{\beta_k\}$ be deterministic scalar sequences such that: 
	\begin{align}\label{conv lemma ineq}
		\EXP{\nu_{k+1}|\nu_0,\dots,\nu_k} \leq (1-\alpha_k)\nu_k + \beta_k \qquad \hbox{for all } k \geq 0,
	\end{align}
	and also there exist some constant $0<\rho<1$, such that for all $k\geq 1$
	\begin{align*}
		0 \leq \alpha_k \leq 1,\ \beta_k \geq 0 \hbox{ and } 
		\frac{\beta_{k-1}}{\alpha_{k-1}} \leq \frac{\beta_k}{\alpha_k}(1+ \rho \alpha_k).
	\end{align*}
	Then, $\EXP{\nu_{k+1}} \leq \frac{\beta_k}{\alpha_k} \tau$, where $\tau \triangleq \max \left \{\frac{\EXP{\nu_1} \alpha_0}{\beta_0}, \frac{1}{1-\rho}\right \}$.
\end{lemma} 
In the following, we make a set of assumptions on the stepsize and the regularization parameter used in Algorithm \ref{algorithm:SMD}. These assumptions will be used in Proposition \ref{conv rate}(a,b) to establish convergence in an almost sure sense and a mean sense. In Proposition \ref{prop condition for sequences}(i), we provide a class of sequences that satisfy all these conditions.
\begin{assumption}\label{assumptions for conv} Assume that for all $k \geq 0$ we have
	
	\noindent$(a)\ \{\gamma_k\}$ and $\{\lambda_k\}$ are positive and non-increasing sequences where $\gamma_0\lambda_0 \leq \frac{L_\omega}{\mu_h}.$    \\  $(b)\ \sum_{k=0}^{\infty}\gamma_k\lambda_k= \infty.$   
	$(c)\ \sum_{k=0}^{\infty}\frac{1}{\gamma_k\lambda_k}\left(\frac{\lambda_{k-1}}{\lambda_k}-1\right)^2<\infty.$ $(d)\ \sum_{k=0}^{\infty} \gamma_k^2 < \infty.$         \\
	$(e)\ \lim_{k\to \infty} \frac{1}{\gamma_k^2\lambda_k^2}\left(\frac{\lambda_{k-1}}{\lambda_k}-1\right)^2=0.$ $(f)\ \lim_{k\to \infty} \frac{\gamma_k}{\lambda_k}=0.$
	
\end{assumption}	
To derive a rate statement in Proposition \ref{conv rate}(c), we make use of the following assumption on the stepsize and regularization parameter. In Proposition \ref{prop condition for sequences}(ii), we will provide an example of the two sequences for which these conditions are met.
\begin{assumption}\label{assumptions for general conv} Assume that for all $k \geq 0$ we have
	
	\noindent	$(a)\ \{\gamma_k\}$ and $\{\lambda_k\}$ are positive and non-increasing sequences where $\gamma_0\lambda_0 \leq \frac{L_\omega}{\mu_h}.$ \\
	$(b)\ $There are a scalar $B_1>0$ and an integer $k_1$ such that $\frac{1}{\gamma_k^3\lambda_k}\left(\frac{\lambda_{k-1}}{\lambda_k}-1\right)^2 \leq B_1$ for $k \geq k_1.$\\
	$(c)\ $There are a scalar $0<\rho<1$ and an integer $k_2$ such that $\frac{\gamma_{k-1}}{\lambda_{k-1}} \leq \frac{\gamma_k}{\lambda_k} \left(1+\rho\frac{\mu_h}{2L_\omega}\gamma_k\lambda_k\right)$ for $k \geq k_2.$\\
	$(d)\ \lim_{k\to \infty} \frac{\gamma_k}{\lambda_k}=0.$
	
\end{assumption}
In the following result, we show convergence of the sequence $\{x_k\}$ generated in Algorithm \ref{algorithm:SMD}. This result will be used in the next sections in order to establish the convergence of the averaging sequence $\bar x_k$ generated by Algorithm \ref{algorithm:SMD} to the optimal solution of the problem \eqref{def:SL}.

\begin{proposition}[{\bf Convergence in almost sure and mean senses for $\mathbf{\{x_k\}}$}]\label{conv rate} Let Assumption \ref{assum:properties} and \ref{assum:RVs} hold. Consider problem \eqref{def:SL} and let $\{x_k\}$ be generated by Algorithm \ref{algorithm:SMD}. Additionally:
	\begin{itemize}
		\item[(a)] Let Assumption \ref{assumptions for conv} hold. Then $D(x_k,x_{\lambda_{k-1}}^*)$ converges to zero almost surely, and
		\begin{align*}
			\lim_{k\to \infty}\EXP{D(x_k,x_{\lambda_{k-1}}^*)}=0.
		\end{align*}
		\item[(b)] Let Assumption \ref{assumptions for conv} hold and $\lim_{k\to \infty} \lambda_k =0$. Then $x_k$ converges to the optimal solution of problem \eqref{def:SL}, i.e., $x_h^*$ almost surely.
		\item[(c)] Let Assumption \ref{assumptions for general conv} hold for some $k_1$ and $k_2$. Then, for all $k \geq \bar{k} \triangleq \max \{k_1,k_2\}$
		\begin{align*}
			\EXP{D(x_{k+1},x_{\lambda_{k}}^*)} \leq  \frac{\gamma_k}{\lambda_k} \tau,
		\end{align*}
		where,
		\begin{align}\label{def theta}
			\tau \triangleq \max \left \{ \frac{2L_\omega M^2 \lambda_{\bar k-1}}{\gamma_{\bar k-1}},\frac{2L_\omega(2C_H^2 L_\omega^3 +4C_F^2\mu_h^3 +4C_H^2 \mu_h^3 \lambda_{\bar k-1}^2)}{\mu_\omega \mu_h^4(1-\rho)} \right \},
		\end{align}
		in which $M$ is such that $\|x\| \leq M$ for all $x \in X$.
	\end{itemize}
\end{proposition}
\begin{proof}
	\noindent (a) Considering relation \eqref{a recursive upper bound}, to show this we apply Lemma \ref{lemma conv lemma}. Let $\nu_k \triangleq D(x_k,x_{\lambda_{k-1}}^*), \alpha_k\triangleq\frac{\mu_h}{2 L_\omega} \gamma_k\lambda_k$ and $\beta_k\triangleq\frac{2C_H^2L_\omega^3}{\mu_h^3\mu_\omega \gamma_k\lambda_k}\left(\frac{\lambda_{k-1}}{\lambda_k}-1\right)^2+\frac{4C_F^2}{\mu_\omega}\gamma_k^2+\frac{4C_H^2}{\mu_\omega}\gamma_k^2\lambda_k^2$. By \eqref{a recursive upper bound} we have
	\begin{align*}
		\EXP{\nu_{k+1}|\nu_1,\dots,\nu_k} \leq (1-\alpha_k)\nu_k + \beta_k \qquad \hbox{for all } k \geq 1.
	\end{align*}
	Note that $ \beta_k \geq0$. Also by Assumption \ref{assumptions for conv}(a), since $\{\gamma_k\}$ and $\{\lambda_k\} $ are positive and $\gamma_0\lambda_0 \leq \frac{2L_\omega}{\mu_h}$ we have $0 \leq \alpha_k \leq 1$. Assumption \ref{assumptions for conv}(b) is sufficient to have $\sum_{k=1}^{\infty}\alpha_k=\infty$. From Assumption \ref{assumptions for conv}(c,d) $\sum_{k=1}^{\infty}\beta_k <\infty$. In addition, we have 
	\begin{align*}
		\lim_{k \to \infty} \frac{\beta_k}{\alpha_k}&=\lim_{k \to \infty} \frac{\frac{2C_H^2L_\omega^3}{\mu_h^3\mu_\omega \gamma_k\lambda_k}\left(\frac{\lambda_{k-1}}{\lambda_k}-1\right)^2+\frac{4C_F^2}{\mu_\omega}\gamma_k^2+\frac{4C_H^2}{\mu_\omega}\gamma_k^2\lambda_k^2}{\frac{\mu_h}{2 L_\omega} \gamma_k\lambda_k}\\
		&=\frac{4C_H^2L_\omega^4}{\mu_h^4\mu_\omega}\lim_{k \to \infty} \frac{1}{\gamma_k^2\lambda_k^2}\left(\frac{\lambda_{k-1}}{\lambda_k}-1\right)^2 +\frac{8C_F^2L_\omega}{\mu_\omega\mu_h}\lim_{k\to \infty} \frac{\gamma_k}{\lambda_k}+
		\frac{8C_H^2 L_\omega}{\mu_\omega\mu_h}\lim_{k\to \infty} \gamma_k\lambda_k.
	\end{align*}
	Applying Assumption \ref{assumptions for conv}(e,f) we only need to prove that $\lim_{k\to \infty} \gamma_k\lambda_k=0$. Since $\{\lambda_k\}$ is non-increasing we have, $\lambda_k \leq\lambda_0$ for all $k\geq 0$. Then $\lambda_k^2 \geq \lambda_0^2$. From Assumption \ref{assumptions for conv}(a) $\gamma_k\lambda_k^2 \leq \gamma_k\lambda_0^2$ implying that $\lambda_0 \frac{\gamma_k}{\lambda_k}$ is an upper bound for $\gamma_k \lambda_k$. So by Assumption \ref{assumptions for conv}(f), $\lim_{k \to \infty} \gamma_k\lambda_k=0$. Consequently $\lim_{k \to \infty} \frac{\beta_k}{\alpha_k}=0$. Therefore, all conditions of Lemma \ref{lemma conv lemma} are met indicating that $D(x_k,x_{\lambda_{k-1}}^*)$ goes to zero almost surely, and $\lim_{k \to \infty} \EXP{D(x_k,x_{\lambda_{k-1}}^*)}=0$.
	
	\noindent (b) Invoking the triangle inequality, we obtain
	\begin{align*}
		\|x_k -x_h^*\|^2 \leq 2\|x_k - x_{\lambda_{k-1}}^* \|^2 + 2 \|x_{\lambda_{k-1}}^*-x_h^*\|^2 \qquad \hbox{for all } k \geq 0.
	\end{align*}
	Using Lemma \ref{D prop}(a), from the preceding inequality we obtain
	\begin{align} \label{rel xk xh}
		\|x_k -x_h^*\|^2 \leq \frac{4}{\mu_\omega}D(x_k,x_{\lambda_{k-1}}^*)+ 2 \|x_{\lambda_{k-1}}^*-x_h^*\|^2 \qquad \hbox{for all } k \geq 0.
	\end{align}
	From Proposition \ref{prop:xk_estimate}(b), we know that when $\lambda_k$ goes to zero, then the sequence $\{x_{\lambda_k}^*\}$ converges to the unique optimal solution of problem \eqref{def:SL}, i.e., $x_h^*$. In addition, from part (a), $D(x_k,x_{\lambda_{k-1}}^*)$ converges to zero almost surely. Therefore, from relation \eqref{rel xk xh}, $\|x_k -x_h^*\|$ converges to zero almost surely.
	
	\noindent (c) We apply Lemma \ref{lemma general conv lemma} to show the desired inequality. Consider relation \eqref{a recursive upper bound}. From Assumption \ref{assumptions for general conv}(b) and that $\{\lambda_k\}$ is non-increasing, we have
	\begin{align*}
		\EXP{D(x_{k+1},x_{\lambda_k}^*)|\mathcal{F}_k}&\leq \left(1-\frac{\mu_h}{2L_\omega}\gamma_k\lambda_k\right)D(x_k,x_{\lambda_{k-1}}^*)\\&+
		\left(\frac{2C_H^2L_\omega^3}{\mu_h^3\mu_\omega}B_1+\frac{4C_F^2}{\mu_\omega}+\frac{4C_H^2}{\mu_\omega}\lambda_{\bar k-1}^2\right)\gamma_k^2,
	\end{align*}
	for all $k\geq \bar k$. For $k \geq \bar k -1$, let us define 
	\begin{align*}
		\nu_k\triangleq D(x_k,x_{\lambda_{k-1}}^*), \ \alpha_k\triangleq \frac{\mu_h}{2 L_\omega} \gamma_k\lambda_k, \hbox{ and } \beta_k\triangleq \left(\frac{2C_H^2L_\omega^3}{\mu_h^3\mu_\omega}B_1+\frac{4C_F^2}{\mu_\omega}+\frac{4C_H^2}{\mu_\omega}\lambda_{\bar k-1}^2\right) \gamma_k^2.
	\end{align*} 
	Now by the preceding inequality we have
	\begin{align*}
		\EXP{\nu_{k+1}|\nu_{\bar k-1},\dots,\nu_k} \leq (1-\alpha_k)\nu_k + \beta_k \qquad \hbox{for all } k \geq \bar k-1.
	\end{align*}
	By Assumption \ref{assumptions for general conv}(a,b), it is easy to see that $0 \leq \alpha_k \leq 1$ and $\beta_k \geq 0$. Also
	\begin{align*}
		\frac{\beta_{k-1}}{\alpha_{k-1}}&= \frac{2 L_\omega\left(\frac{2C_H^2L_\omega^3}{\mu_h^3\mu_\omega}B_1+\frac{4C_F^2}{\mu_\omega}+\frac{4C_H^2}{\mu_\omega}\lambda_{\bar k-1}^2\right)}{\mu_h} \frac{\gamma_{k-1}}{\lambda_{k-1}} \\&\leq \frac{2 L_\omega\left(\frac{2C_H^2L_\omega^3}{\mu_h^3\mu_\omega}B_1+\frac{4C_F^2}{\mu_\omega}+\frac{4C_H^2}{\mu_\omega}\lambda_{\bar k-1}^2\right)}{\mu_h} \frac{\gamma_k}{\lambda_k}\left(1+\rho\frac{\mu_h}{2L_\omega}\gamma_k\lambda_k\right)
		= \frac{\beta_k}{\alpha_k}(1+ \rho \alpha_k),
	\end{align*}
	where we used Assumption \ref{assumptions for general conv}(c) in the second relation. Note that all conditions in Lemma \ref{lemma general conv lemma} are satisfied. Therefore, we can write
	\begin{align}\label{bound for E(D)}
		\EXP{D(x_{k+1},x_{\lambda_{k}}^*)} \leq \frac{2L_\omega(2C_H^2 L_\omega^3 B_1 +4C_F^2\mu_h^3 +4C_H^2 \mu_h^3 \lambda_{\bar k-1}^2)}{\mu_\omega \mu_h^4} \frac{\gamma_k}{\lambda_k} \hat{\tau},
	\end{align}
	where
	\begin{align*}
		\hat{\tau} & \triangleq \max\left\{\frac{\EXP{\nu_{\bar k}}\alpha_{\bar k-1}}{\beta_{\bar k-1}}, \frac{1}{1-\rho}\right\}\\&=\max\left \{\frac{\EXP{\nu_{\bar k}}\mu_h^4\mu_\omega \lambda_{\bar k-1}}{2L_\omega(2C_H^2L_\omega^3 B_1+4C_F^2\mu_h^3+4C_H^2\mu_h^3\lambda_{\bar k-1}^2)\gamma_{\bar k-1}}, \frac{1}{1-\rho}\right \}.
	\end{align*}
	Note that for $\EXP{\nu_{\bar k}}$ by Lemma \ref{D prop}(a) we have
	\begin{align*}
		\EXP{\nu_{\bar k}}= \EXP{D(x_{\bar k},x_{\lambda_{\bar k-1}}^*)} &\leq \EXP{\frac{L_\omega}{2}\|x_{\bar k}-x_{\lambda_{\bar k-1}}^*\|^2}\leq \EXP{\frac{L_\omega}{2}\left(2\|x_{\bar k}\|^2 + 2\|x_{\lambda_{\bar k-1}}^*\|^2\right)}\\ &\leq \EXP{\frac{L_\omega}{2}\left(2M^2 + 2M^2\right)}=2L_\omega M^2,
	\end{align*}
	where $M$ is an upper bound for the set $X$. From the preceding two relations, we have
	\begin{align*}
		\hat{\tau}=\max\left \{\frac{\mu_h^4\mu_\omega M^2 \lambda_{\bar k-1}}{(2C_H^2L_\omega^3 B_1+4C_F^2\mu_h^3+4C_H^2\mu_h^3\lambda_{\bar k-1}^2)\gamma_{\bar k-1}}, \frac{1}{1-\rho} \right \}.
	\end{align*}
	From the preceding relation and definition of $\tau$ in \eqref{def theta}, we have
	\begin{align*}
		\tau = \frac{2L_\omega(2C_H^2 L_\omega^3 B_1 +4C_F^2\mu_h^3 +4C_H^2 \mu_h^3 \lambda_{\bar k-1}^2)}{\mu_\omega \mu_h^4} \hat{\tau}.
	\end{align*}
	From the preceding relation and inequality \eqref{bound for E(D)}, we have
	\begin{align*}
		\EXP{D(x_{k+1},x_{\lambda_{k}}^*)} \leq \frac{\gamma_k}{\lambda_k} \tau. 
	\end{align*}
\end{proof}
Proposition \ref{conv rate} guarantees the convergence of sequence $\{x_k\}$ generated in Algorithm \ref{algorithm:SMD} under general sets of conditions on the stepsize and regularization parameter given by Assumption \ref{assumptions for conv} and \ref{assumptions for general conv}. Below, we provide particular examples of sequences that meet the conditions in Assumption \ref{assumptions for conv} and \ref{assumptions for general conv} and therefore promise the convergences properties stated in Proposition \ref{conv rate}. The proof can be found in Appendix \ref{proof of prop condition for sequences}.

\begin{proposition}[{\bf Feasible sequences for Assumption \ref{assumptions for conv} and \ref{assumptions for general conv}}]\label{prop condition for sequences}
	Assume $\{\gamma_k\}$ and $\{\lambda_k\}$ are sequences such that $\gamma_k=\frac{\gamma_0}{(k+1)^a}$ and $\lambda_k=\frac{\lambda_0}{(k+1)^b}$ where $a,b$ are scalars, $\gamma_0$ and $\lambda_0$ are positive scalars and $\gamma_0\lambda_0 \leq \frac{L_\omega}{\mu_h}$. Then
	\begin{itemize}
		\item[(i)]  The sequences $\{\gamma_k\}$ and $\{\lambda_k\}$ satisfy Assumption \ref{assumptions for conv} when $a,b>0$, $a>b$, $a>0.5$ and $a+b<1$. 
		\item[(ii)] Assumption \ref{assumptions for general conv} is satisfied when $a,b>0$, $a>b$, $a+b<1$ and $3a+b<2$.
	\end{itemize}
\end{proposition}

The following lemma provides sufficient conditions on the weights of an averaging sequence so that the averaging sequence converges. We will make use of this result in Theorem \ref{thm conv for xbar} where we prove convergence of the averaging sequence generated by Algorithm \ref{algorithm:SMD}. 

\begin{lemma}[{\bf Theorem 6, pg.\ 75 of~\cite{Knopp51}}] \label{lemma thm for avr}
	Let $\{u_t\}\subset \mathbf{R}^n$ be a convergent sequence with the limit point $\hat u\in\mathbf{R}^n$ and let $\{\alpha_k\}$ be a
	sequence of positive numbers where $\sum_{k=0}^\infty \alpha_k=\infty$. Suppose $\{v_k\}$ is given by 
	$v_k\triangleq \left({\sum_{t=0}^{k-1} \alpha_t u_t}\right)/{\sum_{t=0}^{k-1} \alpha_t}$ for all $k\ge1$. Then, $\lim_{k \rightarrow \infty} v_k=\hat u.$
\end{lemma}
Below, we present the main result of this section. Theorem \ref{thm conv for xbar} provides convergence in both almost sure and mean senses for the averaging sequence generated by Algorithm \ref{algorithm:SMD} to the unique optimal solution of problem \eqref{def:SL}. Importantly, we specify particular sequences for the stepsize and regularization parameter to establish the convergence properties.
\begin{theorem}[{\bf Convergence in almost sure and mean senses for $\mathbf{\{\bar x_k\}}$}] \label{thm conv for xbar}
	Consider problem \eqref{def:SL}. Let Assumption \ref{assum:properties} and \ref{assum:RVs} hold. Assume $\{\gamma_k\}$ and $\{\lambda_k\}$ are sequences such that $\gamma_k=\frac{\gamma_0}{(k+1)^a}$,  $\lambda_k=\frac{\lambda_0}{(k+1)^b}$ and $\gamma_0$ and $\lambda_0$ are positive scalars and $\gamma_0\lambda_0 \leq \frac{L_\omega}{\mu_h}$. Let $\bar x_k$ be generated by Algorithm \ref{algorithm:SMD}. If $a,b>0$, $a>b$, $a>0.5$, $a+b<1$ and $ar\leq 1$. Then, we have
	\begin{itemize}
		\item[(i)] The sequence $\{\bar x_k\}$ converges to $x_h^*$ almost surely.
		\item[(ii)] We have $\lim_{k \to \infty}\EXP{\|\bar x_{k+1}-x_h^*\|}=0$.
	\end{itemize}
\end{theorem}
\begin{proof}
	\noindent (i) First note that due to the assumptions on scalars $a,b$, all conditions of Proposition \ref{prop condition for sequences}(i) are satisfied, implying that Assumption \ref{assumptions for conv} holds. Note that from definition of $\eta_{t,k} =\gamma_t^r / \sum_{i=0}^{k} \gamma_i^r$ given by \eqref{weighted alg}, it follows that $\sum_{t=0}^{k} \eta_{t,k}=1$. We have,
	\begin{align*}
		\|\bar x_{k+1}-x_h^*\|=\left\|\sum_{t=0}^{k} \eta_{t,k} x_t-\sum_{t=0}^{k}\eta_{t,k} x_h^*\right\|=\left\|\sum_{t=0}^{k} \eta_{t,k}(x_t-x_h^*)\right\|.
	\end{align*}
	Using  the triangular inequality, we obtain
	\begin{align}\label{xbar-xh}
		\|\bar x_{k+1}-x_h^*\|\leq \sum_{t=0}^{k} \eta_{t,k}\|x_t-x_h^*\|.
	\end{align}	
	Now let $\alpha_t \triangleq \gamma_t^r$, $u_t \triangleq \|x_t-x_h^*\|$ and $v_{k+1} \triangleq \sum_{t=0}^{k} \eta_{t,k} \|x_t-x_h^*\|$. Since $ar \leq1$ we can write $\sum_{t=0}^\infty \alpha_t=\sum_{t=0}^\infty \gamma_t^{r}=\sum_{t=0}^\infty (t+1)^{-ar}=\infty$. Since $b>0$, it follows that $\lambda_t=1/(t+1)^b$ goes to zero as $t\to \infty$. So from Proposition \ref{conv rate}(b), $u_t=\|x_t-x_h^*\|$ converges to zero almost surely. Now since conditions of Lemma \ref{lemma thm for avr} are satisfied for $\hat u=0$, we conclude that $\|\bar x_{k+1}-x_h^*\|$ converges to zero almost surely, which means $\{\bar x_k\}$ converges to $x_h^*$ almost surely.
	
	\noindent (ii) 
	Consider relation \eqref{xbar-xh}, we can write
	\begin{align*}
		\EXP{\|\bar x_{k+1}-x_h^*\|}\leq \EXP{\sum_{t=0}^{k} \eta_{t,k}\|x_t-x_h^*\|}=\sum_{t=0}^{k} \eta_{t,k} \EXP{\|x_t-x_h^*\|}.
	\end{align*}	
	Let $\alpha_t \triangleq \gamma_t^r$, $u_t \triangleq \EXP{\|x_t-x_h^*\|}$ and $v_{k+1} \triangleq \sum_{t=0}^{k} \eta_{t,k} \EXP{\|x_t-x_h^*\|}$. To apply Lemma \ref{lemma thm for avr}, we first show that $\{u_t\}$ goes to zero. 
	Adding and subtracting $x^*_{\lambda_{t-1}}$ and using the triangular inequality, we have
	\begin{align*}
		u_t=\EXP{\|x_t -x_h^*\|} &\leq \EXP{\|x_t -x^*_{\lambda_{t-1}}\|+\|x^*_{\lambda_{t-1}}-x_h^*\|}   \\
		&\leq \sqrt{\frac{2}{\mu_\omega}}\EXP{\sqrt{D(x_t,x_{\lambda_{t-1}}^*)}}+ \EXP{\|x_{\lambda_{t-1}}^*-x_h^*\|}\\
		& \leq  \sqrt{\frac{2}{\mu_\omega}} \sqrt{\EXP{D(x_t,x_{\lambda_{t-1}}^*)}}+ \EXP{\|x_{\lambda_{t-1}}^*-x_h^*\|}
	\end{align*}
	where in the second inequality we used $\frac{\mu_\omega}{2}\|x_t -x^*_{\lambda_{t-1}}\|^2 \leq D(x_t,x_{\lambda_{t-1}}^*)$ by Lemma \ref{D prop}(a) and in the third inequality we applied Jensen's inequality for concave functions. From Proposition \ref{prop:xk_estimate}(b), since $\lambda_t$ goes to zero, the sequence $\{x_{\lambda_t}^*\}$ converges to the unique optimal solution of problem \eqref{def:SL}, i.e., $x_h^*$. Moreover, from Proposition \ref{conv rate}(a) we have $\lim_{t\to \infty}\EXP{D(x_t,x_{\lambda_{t-1}}^*)}=0$. Therefore, $\lim_{t\to \infty}u_t=\EXP{\|x_t -x_h^*\|}=0$. The remainder of the proof can be done in a similar vein as part (i) through applying Lemma \ref{lemma thm for avr}.
\end{proof}

\section{Rate analysis}
Our goal in this section is to provide complexity analysis for the developed  IR-SMD method. To this end, first in Lemma \ref{lemma averaging}, we derive an upper bound in terms of the objective function of problem \eqref{def:firstlevel}, i.e., $f$, evaluated at the averaging sequence $\{\bar x_N\}$ generated by Algorithm \ref{algorithm:SMD}. This result will then be employed in Theorem \ref{thm rate} to derive a rate statement for problem \eqref{def:firstlevel}. 

\label{sec:rate}
\begin{lemma} [\bf{An error bound for problem \eqref{def:firstlevel}}] \label{lemma averaging} Consider the sequence $\{\bar x_N\}$ generated by Algorithm \ref{algorithm:SMD}. Let Assumption \ref{assum:properties} and \ref{assum:RVs} hold. Also let $\{\gamma_k\}$ and $\{\lambda_k\}$ be positive and non-increasing sequences. Then, for all $N\geq 1$ and $z\in X$ we have
	\begin{align*}
		\EXP{f(\bar x_N)}-f(z)\leq &\left(\sum_{k=0}^{N-1} \gamma_k^r\right)^{-1}\left(2L_\omega M^2 \left(\gamma_{N-1}^{r-1}+\gamma_0^{r-1}\right)+ 2M_h\sum_{k=0}^{N-1}\gamma_k^r\lambda_k \right. \cr\\& \left. + \frac{C_F^2+C_H^2 \lambda_0^2}{\mu_\omega} \sum_{k=0}^{N-1}  \gamma_k^{r+1}\right),
	\end{align*}
	where $M_h$ is an upper bound for function $h$ over the set $X$ and $M$ is an upper bound for the set $X$.
\end{lemma}
\begin{proof}
	Let $k\geq1 $ be given. Along similar lines to the beginning of Lemma \ref{recursive bd}, we can write (see relation \eqref{proof rate1}) 
	\begin{align}\label{error bound1}
		\langle \gamma_k(g_F(x_k,\xi_k) + \lambda_k g_H(x_k, \tilde \xi_k)) + \nabla \omega(x_{k+1})-\nabla \omega(x_k),z-x_{k+1} \rangle \geq 0  \ \hbox{for all}, \ z\in X.
	\end{align}
	By properties of function $D$ in Lemma \ref{D prop}(a,b), we have
	\begin{align*}
		\langle \nabla \omega(x_{k+1})-\nabla \omega(x_k) ,z-x_{k+1} \rangle&= D(x_k,z)-D(x_{k+1},z)-D(x_k,x_{k+1}),\\
		D(x_k,x_{k+1}) &\geq \frac{\mu_\omega}{2} \|x_{k+1}-x_k\|^2.
	\end{align*}
	Combining the preceding two relations with inequality \eqref{error bound1}, and rearranging the terms we obtain
	\begin{align*}
		D(x_{k+1},z) - D(x_k,z) & \leq -\frac{\mu_\omega}{2} \|x_{k+1}-x_k\|^2 \\& +\gamma_k\underbrace{\langle g_F(x_k,\xi_k),z-x_{k+1} \rangle}_\text{Term1}+ \gamma_k\lambda_k\underbrace{\langle g_H(x_k, \tilde \xi_k),z-x_{k+1} \rangle}_\text{Term2}.
	\end{align*}
	By adding and subtracting $x_k$ in Term1 and Term2, we obtain
	\begin{align*}
		D(x_{k+1},z) - D(x_k,z) &\leq -\frac{\mu_\omega}{2} \|x_{k+1}-x_k\|^2 +\gamma_k\langle g_F(x_k,\xi_k),z-x_k \rangle\\&+\underbrace{\gamma_k\langle g_F(x_k,\xi_k),x_k-x_{k+1} \rangle}_\text{Term3}+ \gamma_k\lambda_k\langle g_H(x_k, \tilde \xi_k),z-x_k \rangle \\&+ \underbrace{\gamma_k\lambda_k\langle g_H(x_k, \tilde \xi_k),x_k-x_{k+1} \rangle}_\text{Term4}.
	\end{align*}
	By multiplying and dividing Term3 and Term4 by $\sqrt{\frac{2}{\mu_\omega}}$ and then applying Fenchel's inequality, i.e., $\langle a,b \rangle \leq \frac{1}{2}\|a\|^2 +\frac{1}{2}\|b\|_*^2$, we have
	\begin{align*}
		D(x_{k+1},z) - D(x_k,z)&\leq -\frac{\mu_\omega}{2} \|x_{k+1}-x_k\|^2 +
		\gamma_k\langle g_F(x_k,\xi_k),z-x_k\rangle\\&+ \frac{\gamma_k^2}{\mu_\omega}\|g_F(x_k,\xi_k)\|_*^2 + \frac{\mu_\omega}{4}\|x_{k+1}-x_k\|^2+
		\gamma_k\lambda_k\langle g_H(x_k, \tilde \xi_k),z-x_k \rangle\\&+ \frac{\gamma_k^2\lambda_k^2}{\mu_\omega}\|g_H(x_k, \tilde \xi_k)\|_*^2 + \frac{\mu_\omega}{4}\|x_{k+1}-x_k\|^2.
	\end{align*}
	Therefore, we obtain
	\begin{align*}
		D(x_{k+1},z) - D(x_k,z) &\leq \frac{\gamma_k^2}{\mu_\omega}\|g_F(x_k,\xi_k)\|_*^2+\frac{\gamma_k^2\lambda_k^2}{\mu_\omega}\|g_H(x_k, \tilde \xi_k)\|_*^2\\&+\gamma_k\langle g_F(x_k,\xi_k),z-x_k\rangle+ \gamma_k\lambda_k\langle g_H(x_k, \tilde \xi_k),z-x_k \rangle .
	\end{align*}
	By taking conditional expectations from both sides and considering Assumption \ref{assum:properties}(d,e), we have
	\begin{align}\label{error bound2}
		\EXP{D(x_{k+1},z)|\mathcal{F}_k}-D(x_k,z)& \leq \frac{\gamma_k^2}{\mu_\omega}C_F^2+\frac{\gamma_k^2\lambda_k^2}{\mu_\omega}C_H^2 \\&+\gamma_k\langle g_f(x_k),z-x_k\rangle+ \gamma_k\lambda_k\langle g_h(x_k),z-x_k \rangle. \nonumber
	\end{align}
	From the definition of subgradient for functions $f$ and $h$ we have $\langle g_f(x), y-x\rangle  \leq f(y)- f(x)$ and $\langle g_h(x), y-x\rangle  \leq h(y)- h(x)$ for all $x,y \in X$. From these relations and \eqref{error bound2} we obtain
	\begin{align*}
		\EXP{D(x_{k+1},z)|\mathcal{F}_k}-D(x_k,z) &\leq \frac{\gamma_k^2}{\mu_\omega}C_F^2+\frac{\gamma_k^2\lambda_k^2}{\mu_\omega}C_H^2\\&+\gamma_k(f(z)-f(x_k))+ \gamma_k\lambda_k\left(h(z)-h(x_k)\right).
	\end{align*}
	Rearranging the terms, we have
	\begin{align*}
		\EXP{D(x_{k+1},z)|\mathcal{F}_k}-D(x_k,z)& \leq \frac{\gamma_k^2}{\mu_\omega}C_F^2+\frac{\gamma_k^2\lambda_k^2}{\mu_\omega}C_H^2\\&+\gamma_k(f(z)+\lambda_k h(z))- \gamma_k(f(x_k)+\lambda_k h(x_k)).
	\end{align*}
	Taking expectations from both sides and applying the law of total expectation, we have
	\begin{align} \label{error bound3}
		\EXP{D(x_{k+1},z)}-\EXP{D(x_k,z)} &\leq \frac{\gamma_k^2}{\mu_\omega}C_F^2+\frac{\gamma_k^2\lambda_k^2}{\mu_\omega}C_H^2\\&+\gamma_k(f(z)+\lambda_k h(z))- \gamma_k\left(\EXP{f(x_k)+\lambda_k h(x_k)}\right).\nonumber
	\end{align}
	Multiplying both sides of the preceding inequality by $\gamma_k^{r-1}$ and adding and subtracting $\gamma_{k-1}^{r-1}\EXP{D(x_k,z)}$ to the left-hand side, we have
	\begin{align}\label{error bound4}
		&\gamma_k^{r-1}\EXP{D(x_{k+1},z)}-\gamma_{k-1}^{r-1}\EXP{D(x_k,z)}-(\gamma_k^{r-1}-\gamma_{k-1}^{r-1})\EXP{D(x_k,z)}\leq\frac{\gamma_k^{r+1}}{\mu_\omega}C_F^2\\&+\frac{\gamma_k^{r+1}\lambda_k^2}{\mu_\omega}C_H^2+\gamma_k^r(f(z)+\lambda_k h(z))- \gamma_k^r(\EXP{f(x_k)+\lambda_k h(x_k)}).\nonumber
	\end{align}
	Note that since $r<1$ and $\gamma_k$ is non-increasing, $\gamma_k^{r-1}-\gamma_{k-1}^{r-1}$ is nonnegative. Also from Lemma \ref{D prop}(a) for $\EXP{D(x_k,z)}$ we have
	\begin{align}\label{error bound5}
		\EXP{D(x_k,z)} &\leq \EXP{\frac{L_\omega}{2}\|x_k-z\|^2}\leq \EXP{\frac{L_\omega}{2}\left(2\|x_k\|^2 + 2\|z\|^2\right)}\\ &\leq \EXP{\frac{L_\omega}{2}\left(2M^2 + 2M^2\right)}=2L_\omega M^2,\nonumber
	\end{align}
	where $M$ is an upper bound for our set $X$. From the bound given by \eqref{error bound5} and relation \eqref{error bound4} we obtain
	\begin{align*}
		&\gamma_k^{r-1}\EXP{D(x_{k+1},z)}-\gamma_{k-1}^{r-1}\EXP{D(x_k,z)}-(\gamma_k^{r-1}-\gamma_{k-1}^{r-1})2L_\omega M^2\leq\frac{\gamma_k^{r+1}}{\mu_\omega}C_F^2\\&+\frac{\gamma_k^{r+1}\lambda_k^2}{\mu_\omega}C_H^2+\gamma_k^r(f(z)+\lambda_k h(z))- \gamma_k^r(\EXP{f(x_k)+\lambda_k h(x_k)}).\nonumber
	\end{align*}
	Summing the preceding inequalities over $k=1,2, \cdots,N-1$, we obtain
	\begin{align*}
		&\gamma_{N-1}^{r-1}\EXP{D(x_{N},z)}-\gamma_{0}^{r-1}\EXP{D(x_1,z)}-(\gamma_{N-1}^{r-1}-\gamma_{0}^{r-1})2L_\omega M^2 \leq \frac{C_F^2}{\mu_\omega} \sum_{k=1}^{N-1} \gamma_k^{r+1} \\& +\frac{C_H^2}{\mu_\omega} \sum_{k=1}^{N-1} \gamma_k^{r+1}\lambda_k^2+ \sum_{k=1}^{N-1}\gamma_k^r(f(z)+\lambda_k h(z))-\sum_{k=1}^{N-1} \gamma_k^r\left(\EXP{f(x_k)+\lambda_k h(x_k)}\right).\nonumber
	\end{align*}
	By removing the nonnegative terms $\gamma_{N-1}^{r-1}\EXP{D(x_{N},z)}$ and $2L_\omega M^2\gamma_{0}^{r-1}$ from the left-hand side of the preceding inequality, we have
	\begin{align}\label{sum overk}
		-\gamma_{0}^{r-1}\EXP{D(x_1,z)}&-2L_\omega M^2\gamma_{N-1}^{r-1} \leq \frac{C_F^2}{\mu_\omega} \sum_{k=1}^{N-1} \gamma_k^{r+1} +\frac{C_H^2}{\mu_\omega} \sum_{k=1}^{N-1} \gamma_k^{r+1}\lambda_k^2\\&+ \sum_{k=1}^{N-1}\gamma_k^r(f(z)+\lambda_k h(z))-\sum_{k=1}^{N-1} \gamma_k^r(\EXP{f(x_k)+\lambda_k h(x_k)}).\nonumber
	\end{align}
	For $\EXP{D(x_1,z)}$ from \eqref{error bound3} when $k=0$, we have
	\begin{align} \label{inequality for d1}
		\EXP{D(x_1,z)}&\leq \frac{\gamma_0^2}{\mu_\omega}C_F^2+\frac{\gamma_0^2 \lambda_0^2}{\mu_\omega}C_H^2+\gamma_0(f(z)+\lambda_0 h(z))\\&-\gamma_0\left(\EXP{f(x_0)}+\lambda_0 \EXP{h(x_0)} \right)+2L_\omega M^2, \nonumber
	\end{align}
	where we substituted $\EXP{D(x_0,z)}$ by $2L_\omega M^2$ as we showed in \eqref{error bound5}. Then, by multiplying both sides of \eqref{inequality for d1} by $\gamma_0^{r-1}$ and summing the resulting relation with \eqref{sum overk} and rearranging the terms, we have
	\begin{align*}
		&\sum_{k=0}^{N-1} \gamma_k^r\left(\EXP{f(x_k)+\lambda_k h(x_k)}\right)-\sum_{k=0}^{N-1}\gamma_k^r(f(z)+\lambda_k h(z)) \leq 2L_\omega M^2 \left(\gamma_{N-1}^{r-1}+\gamma_0^{r-1}\right)\\ &+ \frac{C_F^2}{\mu_\omega} \sum_{k=0}^{N-1} \gamma_k^{r+1} +\frac{C_H^2}{\mu_\omega} \sum_{k=0}^{N-1} \gamma_k^{r+1}\lambda_k^2.
	\end{align*}
	Dividing both sides by $\sum_{k=0}^{N-1} \gamma_k^r$ and considering the definition of $\eta_{k,N-1}$ given by \eqref{weighted alg}, we obtain
	\begin{align*}
		&\sum_{k=0}^{N-1} \eta_{k,N-1}\left(\EXP{f(x_k)+\lambda_k h(x_k)}\right) -\sum_{k=0}^{N-1}\eta_{k,N-1}(f(z)+\lambda_k h(z))\\&\leq \left(\sum_{k=0}^{N-1} \gamma_k^r\right)^{-1}\left(2L_\omega M^2 \left(\gamma_{N-1}^{r-1}+\gamma_0^{r-1}\right)+  \frac{C_F^2}{\mu_\omega} \sum_{k=0}^{N-1}  \gamma_k^{r+1} +\frac{C_H^2}{\mu_\omega} \sum_{k=0}^{N-1} \gamma_k^{r+1} \lambda_k^2\right).
	\end{align*}
	So we can write
	\begin{align*}
		&\EXP{\sum_{k=0}^{N-1} \eta_{k,N-1}(f(x_k)+\lambda_k h(x_k))} -\sum_{k=0}^{N-1}\eta_{k,N-1}(f(z)+\lambda_k h(z))\\&\leq \left(\sum_{k=0}^{N-1} \gamma_k^r\right)^{-1}\left(2L_\omega M^2 \left(\gamma_{N-1}^{r-1}+\gamma_0^{r-1}\right)+  \frac{C_F^2}{\mu_\omega} \sum_{k=0}^{N-1}  \gamma_k^{r+1} +\frac{C_H^2}{\mu_\omega} \sum_{k=0}^{N-1} \gamma_k^{r+1} \lambda_k^2\right).
	\end{align*}
	Rearranging the terms, we have
	\begin{align*}
		&\EXP{\sum_{k=0}^{N-1} \eta_{k,N-1}f(x_k)} -\sum_{k=0}^{N-1}\eta_{k,N-1}f(z) \leq \sum_{k=0}^{N-1}\eta_{k,N-1}\lambda_k h(z)-\EXP{\sum_{k=0}^{N-1} \eta_{k,N-1}\lambda_k h(x_k)}\\&+ \left(\sum_{k=0}^{N-1} \gamma_k^r\right)^{-1}\left(2L_\omega M^2 \left(\gamma_{N-1}^{r-1}+\gamma_0^{r-1}\right)+  \frac{C_F^2}{\mu_\omega} \sum_{k=0}^{N-1}  \gamma_k^{r+1} +\frac{C_H^2}{\mu_\omega} \sum_{k=0}^{N-1} \gamma_k^{r+1} \lambda_k^2\right).
	\end{align*}
	In left-hand side, note that we have $\sum_{k=0}^{N-1} \eta_{k,N-1}=1$ from definition of $\eta_{k,N-1}$ in \eqref{weighted alg}. So convexity of $f$ implies that $f(\bar x_N) \leq \sum_{k=0}^{N-1} \eta_{k,N-1}f(x_k)$. We have
	\begin{align} \label{error bound6}
		&\EXP{f(\bar x_N)}-f(z)  \leq \underbrace{\sum_{k=0}^{N-1}\eta_{k,N-1}\lambda_k h(z)-\EXP{\sum_{k=0}^{N-1} \eta_{k,N-1}\lambda_k h(x_k)}}_\text{\hbox{Term5}}\\&+ \left(\sum_{k=0}^{N-1} \gamma_k^r\right)^{-1}\left(2L_\omega M^2 \left(\gamma_{N-1}^{r-1}+\gamma_0^{r-1}\right)+  \frac{C_F^2}{\mu_\omega} \sum_{k=0}^{N-1}  \gamma_k^{r+1} +\frac{C_H^2}{\mu_\omega} \sum_{k=0}^{N-1} \gamma_k^{r+1} \lambda_k^2\right).\nonumber
	\end{align}
	For Term5 we can write
	\begin{align*}
		\hbox{Term5}&=\EXP{\sum_{k=0}^{N-1}\eta_{k,N-1}\lambda_k h(z)-\sum_{k=0}^{N-1} \eta_{k,N-1}\lambda_k h(x_k)} \\& \leq \EXP{\sum_{k=0}^{N-1}\eta_{k,N-1}\lambda_k |h(z)-h(x_k)|} \leq 2M_h\sum_{k=0}^{N-1}\eta_{k,N-1}\lambda_k,
	\end{align*}
	where we used the definition of $M_h$. Using this inequality for Term5 and \eqref{error bound6} we obtain,
	\begin{align*}
		&\EXP{f(\bar x_N)}-f(z) \leq 2M_h\sum_{k=0}^{N-1}\eta_{k,N-1}\lambda_k \\&+ \left(\sum_{k=0}^{N-1} \gamma_k^r\right)^{-1}\left(2L_\omega M^2 \left(\gamma_{N-1}^{r-1}+\gamma_0^{r-1}\right)+  \frac{C_F^2}{\mu_\omega} \sum_{k=0}^{N-1}  \gamma_k^{r+1} +\frac{C_H^2}{\mu_\omega} \sum_{k=0}^{N-1} \gamma_k^{r+1} \lambda_k^2\right).
	\end{align*}
	Applying the formula of $ \eta_{k,N-1}$ in \eqref{weighted alg}, we obtain
	\begin{align*}
		\EXP{f(\bar x_N)}-f(z) \leq \left(\sum_{k=0}^{N-1} \gamma_k^r\right)^{-1}& \left(2L_\omega M^2 \left(\gamma_{N-1}^{r-1}+\gamma_0^{r-1}\right)+ 2M_h\sum_{k=0}^{N-1}\gamma_k^r\lambda_k \right. \cr\\& \left.+ \frac{C_F^2}{\mu_\omega} \sum_{k=0}^{N-1}  \gamma_k^{r+1} +\frac{C_H^2}{\mu_\omega} \sum_{k=0}^{N-1} \gamma_k^{r+1} \lambda_k^2\right).
	\end{align*}
	Since $\{\lambda_k\}$ is a non-increasing sequence, we obtain
	\begin{align*}
		\EXP{f(\bar x_N)}-f(z)\leq \left(\sum_{k=0}^{N-1} \gamma_k^r\right)^{-1} & \left(2L_\omega M^2 \left(\gamma_{N-1}^{r-1}+\gamma_0^{r-1}\right)+ 2M_h\sum_{k=0}^{N-1}\gamma_k^r\lambda_k+ \right. \cr\\& \left. \frac{C_F^2+C_H^2 \lambda_0^2}{\mu_\omega} \sum_{k=0}^{N-1}  \gamma_k^{r+1}\right).
	\end{align*}
\end{proof}
To derive the convergence rate statement in Theorem \ref{thm rate}, we make use of the following result (see Lemma 9, page 418 in \cite{Farzad3}).

\begin{lemma}\label{lemma:ineqHarmonic}
	For any scalar $\alpha\neq -1$ and integers $\ell$ and $N$ where $0\leq \ell \leq N-1$, we have
	\begin{align*}
		\frac{N^{\alpha+1}-(\ell+1)^{\alpha+1}}{\alpha+1}\leq \sum_{k=\ell}^{N-1}(k+1)^\alpha \leq (\ell+1)^\alpha+\frac{(N+1)^{\alpha+1}-(\ell+1)^{\alpha+1}}{\alpha+1}.
	\end{align*}
\end{lemma}

In the following result, we show that using Algorithm \ref{algorithm:SMD}, and under specific choices for the stepsize and regularization sequences, the objective function of problem \eqref{def:firstlevel} converges to its optimal value in a near optimal rate of $\mathcal{O}\left(1/N^{0.5-\delta}\right)$, where $\delta>0$ is an arbitrary small number. We also establish almost sure convergence of the generated sequence by Algorithm \ref{algorithm:SMD} to the unique optimal solution of problem \eqref{def:SL}.

\begin{theorem} [\bf{Convergence and a rate statement for Algorithm \ref{algorithm:SMD}}] \label{thm rate}
	Consider problem \eqref{def:SL}. Let Assumption \ref{assum:properties} and \ref{assum:RVs} hold. Let $\{\bar x_N\}$ be generated by Algorithm \ref{algorithm:SMD}. Let $0<\delta<0.5$ be an arbitrary scalar and $r<1$ be an arbitrary constant. Assume for $0<\delta<0.5$, $\{\gamma_k\}$ and $\{\lambda_k\}$ are sequences such that
	\begin{align*}
		\boxed{
			\gamma_k=\frac{\gamma_0}{(k+1)^{0.5+0.5\delta}} \hbox{ and } \lambda_k=\frac{\lambda_0}{(k+1)^{0.5-\delta}},}
	\end{align*}
	where $\gamma_0$ and $\lambda_0$ are positive scalars and $\gamma_0\lambda_0 \leq \frac{L_\omega}{\mu_h}$. Then,
	\begin{itemize}
		\item[(i)] The sequence $\{\bar x_N\}$ converges to $x_h^*$ almost surely.
		\item[(ii)] We have $\lim_{N \to \infty}\EXP{\|\bar x_{N+1}-x_h^*\|}=0$.
		\item[(iii)] $\EXP{f(\bar x_N)}$ converges to $f^*$ with the rate of ${\cal O}\left(1/N^{0.5-\delta}\right)$, where $f^*$ is the optimal objective value of problem \eqref{def:firstlevel}.
	\end{itemize}
\end{theorem}
\begin{proof} Throughout, we use the notation $a=0.5+0.5\delta$, $b=0.5-\delta$.
	
	\noindent (i,ii) From the values of $a$ and $b$, and that $r<1$ and $0<\delta<0.5$, we have
	\begin{align*}
		a>b>0,\ a>0.5, \ a+b=1-0.5\delta<1,\ ar= 0.5(1+\delta)r <0.5(1.5)=0.75<1.
	\end{align*}
	This implies that all conditions of Theorem \ref{thm conv for xbar} are satisfied. Therefore, $\{\bar x_N\}$ converges to $x_h^*$ almost surely and $\lim_{N \to \infty}\EXP{\|\bar x_{N+1}-x_h^*\|}=0$.
	
	\noindent (iii) Substituting $\gamma_k$ and $\lambda_k$ in the inequality given by Lemma \ref{lemma averaging} and selecting $z=x^*$, we obtain
	\begin{align*}
		\EXP{f(\bar x_N)}-f^*\leq & \left(\sum_{k=0}^{N-1} \frac{\gamma_0^r}{(k+1)^{ar}}\right)^{-1}\left(2L_\omega M^2\gamma_0^{r-1} \left(N^{a(1-r)}+1\right) \right. \cr\\& \left. + 2M_h\sum_{k=0}^{N-1}\frac{\gamma_0^r\lambda_0}{(k+1)^{ar+b}}+  \left( \frac{C_F^2+C_H^2 \lambda_0^2}{\mu_\omega} \right) \sum_{k=0}^{N-1} \frac{\gamma_0^{r+1}}{(k+1)^{a(r+1)}}\right),
	\end{align*}
	where $M_h$ is the upper bound for function $h$ over the set $X$. Note that since $a>b>0$, we have $a(r+1)>ar+b$. Thus, $(k+1)^{a(r+1)}> (k+1)^{ar+b}$. Taking this into account, from the preceding relation we have
	\begin{align}\label{boundonf}
		\EXP{f(\bar x_N)}-f^*\leq & \left(\sum_{k=0}^{N-1} \frac{\gamma_0^r}{(k+1)^{ar}}\right)^{-1}  \left(2L_\omega M^2\gamma_0^{r-1} \left(N^{a(1-r)}+1\right) \right. \cr  \\ & \left. + \gamma_0^r \left(2M_h\lambda_0+ \left( \frac{C_F^2\gamma_0+C_H^2 \lambda_0^2\gamma_0}{\mu_\omega} \right)\right) \sum_{k=0}^{N-1}\frac{1}{(k+1)^{ar+b}}\right). \nonumber
	\end{align}
	Let us consider the following definitions.
	\begin{align*}
		&\hbox{Term1}= \left(\sum_{k=0}^{N-1} \frac{1}{(k+1)^{ar}}\right)^{-1},\hbox{ Term2}= \left(\sum_{k=0}^{N-1} \frac{1}{(k+1)^{ar}}\right)^{-1} N^{a(1-r)},\\
		&\hbox{Term3}=\left(\sum_{k=0}^{N-1} \frac{1}{(k+1)^{ar}}\right)^{-1}\left(\sum_{k=0}^{N-1}\frac{1}{(k+1)^{ar+b}}\right).
	\end{align*}
	Using these definitions, equivalently from \eqref{boundonf}, we have
	\begin{align} \label{terms}
		\EXP{f(\bar x_N)}-f^* & \leq 2L_\omega M^2\gamma_0^{-1} (\hbox{Term1}+\hbox{Term2}) \\&+ \left(2M_h\lambda_0 + \left( \frac{C_F^2\gamma_0+C_H^2 \lambda_0^2\gamma_0}{\mu_\omega} \right)\right) \hbox{Term3}.
	\end{align}
	Next, we estimate the terms 1, 2 and 3. Note that for given $0<\delta<0.5$, from the definitions of $a$ and $b$ and that $r<1$, we have $ar<1$ and $ar+b<a+b<1$. By applying Lemma \ref{lemma:ineqHarmonic}, we have
	\begin{align*}
		&\hbox{Term1} \leq\frac{1-ar}{N^{1-ar}-1}={\cal O}\left(N^{-(1-ar)}\right), \\
		&\hbox{Term2} \leq\frac{N^{a(1-r)}}{\frac{N^{1-ar}-1}{1-ar}}=\frac{(1-ar)N^{a(1-r)}}{N^{1-ar}-1}={\cal O}\left(N^{-(1-a)}\right), \\
		&\hbox{Term3} \leq\frac{\frac{(N+1)^{1-ar-b}-1}{1-ar-b}+1}{\frac{N^{1-ar}-1}{1-ar}} =\frac{(1-ar)\left((N+1)^{1-ar-b}-1\right)}{(1-ar-b)\left(N^{1-ar}-1\right)} + \frac{1-ar}{N^{1-ar}-1} \\& ={\cal O}\left(N^{-b}\right) + {\cal O}\left(N^{-(1-ar)}\right).
	\end{align*}
	From the preceding bounds and relation \eqref{terms}, we have
	\begin{align*}
		\EXP{f(\bar x_N)}-f^*\leq {\cal O}\left(N^{-\min \{1-ar,1-a,b\}}\right)= {\cal O}\left(N^{-\min \{1-a,b\}}\right),
	\end{align*}
	where we used $1-a\leq 1-ar$. Replacing $a$ and $b$ by their values, we have
	\begin{align*}
		\EXP{f(\bar x_N)}-f^*\leq {\cal O}\left(N^{-\min \{0.5-0.5\delta,0.5-\delta\}}\right)={\cal O}\left(N^{-(0.5-\delta)}\right).
	\end{align*}
\end{proof}

\section{Experimental results}
\label{sec:num}
In this section, we examine the performance of Algorithm \ref{algorithm:SMD} on different sets of problems. First we consider linear inverse problems \cite{Beck09}. Given a matrix $A\in \mathbf{R}^{m\times n}$ and a vector $b\in \mathbf{R}^m$, the goal is to find $x \in \mathbf{R}^n$ such that $Ax+\delta=b$, where $\delta \in \mathbf{R}^m$ denotes an unknown noise. To solve this problem, one may consider the following least-squares problem given as 
\begin{equation}\label{def:firstlevelLS}
\begin{split}
&\begin{array}{ll}
\hbox{minimize } &   f(x)\triangleq \|Ax-b\|_2^2\cr
\hbox{subject to } &  x \in \mathbf{R}^n,
\end{array}
\end{split}
\end{equation}
In many applications arising from signal processing \cite{Friedlander07} and image reconstruction \cite{Garrigos17}, problem \eqref{def:firstlevelLS} is ill-posed. To address ill-posedness, we consider a \eqref{def:SL} model here, where in as objective function $h$, we minimize a desired regularizer. Here we assume this function is characterized by both $\ell_1$ and $\ell_2$ norms. More precisely, we consider the following model
\begin{align} \label{numeric upper}
	\mbox{minimize}\ \ &h(x)\triangleq  \frac{\mu_h }{2}\|x\|_2^2+ \|x\|_1\\
	\hbox{subject to}\ \ &x \in \argmin_{y \in \mathbf R^n } \|Ay-b\|_2^2, \nonumber
\end{align}
where $\mu_{h}>0$ is the strongly convex parameter. Note that in contrast with the work in \cite{ Beck14,Sabach17}, the function $h$ is nondifferentiable. To perform numerical experiments, similar to \cite{ Beck14,Sabach17}, we consider three inverse problems, namely ``Baart'', ``Philips'', and ``Foxgood'' where they differ in terms of the underlying method to generate $A$ and $b$. More information on these problems can be found on the website {http://www2.imm.dtu.dk/~pcha/Regutools/}. Table \ref{table:inverseProblem} summarizes the results of our experiments. Here for each class of the three inverse models, we vary the value of the initial point $x_0$ and the dimension $n$. We report the value of $|f(\bar x_N)-f^*|$, referred to as the feasibility gap, and the value of $|h(\bar x_N)-h^*|$, referred to as the optimality gap. For all different scenarios of these experiments, we let the algorithm stop after 250 seconds. We let $\mu_{h}=0.5$ in our experiments  and let $\g_k$ and $\lambda_k$ be given by the update rules in Theorem \ref{thm rate}. To evaluate $f^*$ and $h^*$, we consider a representation of problem \eqref{def:firstlevel} as the linear system equation $Ax=b$. Using this reformulation, we were able to evaluate both $f^*$ and $h^*$ directly using the \textit{quadprog} package in Matlab. We observe that while the feasibility gap is  very small, the optimality gap has approached to zero for almost all the scenarios. We note that for Foxgood when $n=500$ and $1000$, even though the optimality gap is not negligible, the relative optimality gap is as small as $0.4$\% and $0.6$\%, respectively. 

\begin{table}[t]
	\caption{Performance of IR-SMD method for linear inverse problems}
	\centering
	\scalebox{0.9}{
		\label{determinstic}
		\begin{tabular}{c|c|ccc|ccc}
			\hline\noalign{\smallskip}
			\multirow{2}{*}{\begin{tabular}[c]{@{}c@{}}Initial Point\\ $x_0$\end{tabular}} & \multirow{2}{*}{$n$} & \multicolumn{3}{c|}{\begin{tabular}[c]{@{}c@{}}Feasibility Gap\\ $|f(\bar x_N)-f^*|$\end{tabular}} & \multicolumn{3}{c}{\begin{tabular}[c]{@{}c@{}}Optimality Gap\\ $|h(\bar x_N)-h^*|$\end{tabular}} \\ \cline{3-8} 
			&                      & Baart                         & Foxgood                       & Phillips                      & Baart                       & Foxgood                       & Phillips                      \\ \hline\noalign{\smallskip}
			\multirow{5}{*}{$x_0=-10\times\mathbf{1}_n$ }                                               & 20                   & $3.12\mathrm{e}{-7}$          & $3.48\mathrm{e}{-6}$          & $7.87\mathrm{e}{-9}$          & 0.01                        & 0.07                          & 0.00                          \\
			& 100                  & $1.26\mathrm{e}{-6}$          & $1.28\mathrm{e}{-5}$          & $2.95\mathrm{e}{-8}$          & 0.02                        & 0.19                          & 0.00                          \\
			& 200                  & $4.04\mathrm{e}{-6}$          & $3.23\mathrm{e}{-5}$          & $7.96\mathrm{e}{-8}$          & 0.04                        & 0.41                          & 0.00                          \\
			& 500                  & $3\mathrm{e}{-5}$             & $1.48\mathrm{e}{-4}$          & $4.8\mathrm{e}{-7}$           & 0.18                        & 1.22                          & 0.00                          \\
			& 1000                 & $2.46\mathrm{e}{-4}$          & $8.49\mathrm{e}{-4}$          & $3.59\mathrm{e}{-6}$          & 0.74                        & 3.5                           & 0.01                          \\ \hline\noalign{\smallskip}
			\multirow{5}{*}{$x_0=\mathbf{0}_n$ }                                                  & 20                   & $3.15\mathrm{e}{-7}$          & $3.47\mathrm{e}{-6}$          & $7.84\mathrm{e}{-9}$          & 0.01                        & 0.07                          & 0.00                          \\
			& 100                  & $1.29\mathrm{e}{-6}$          & $1.26\mathrm{e}{-5}$          & $2.94\mathrm{e}{-8}$          & 0.02                        & 0.19                          & 0.00                          \\
			& 200                  & $4\mathrm{e}{-6}$             & $3.29\mathrm{e}{-5}$          & $7.99\mathrm{e}{-8}$          & 0.04                        & 0.41                          & 0.00                          \\
			& 500                  & $2.91\mathrm{e}{-5}$          & $1.45\mathrm{e}{-4}$          & $4.97\mathrm{e}{-7}$          & 0.18                        & 1.22                          & 0.00                          \\
			& 1000                 & $2.5\mathrm{e}{-4}$           & $8.3\mathrm{e}{-4}$           & $3.64\mathrm{e}{-6}$          & 0.75                        & 3.47                          & 0.01                          \\ \hline\noalign{\smallskip}
			\multirow{5}{*}{$x_0=10\times\mathbf{1}_n$ }                                                & 20                   & $3.15\mathrm{e}{-7}$          & $3.47\mathrm{e}{-6}$          & $7.96\mathrm{e}{-9}$          & 0.01                        & 0.07                          & 0.00                          \\
			& 100                  & $1.27\mathrm{e}{-6}$          & $1.26\mathrm{e}{-5}$          & $2.94\mathrm{e}{-8}$          & 0.02                        & 0.19                          & 0.00                          \\
			& 200                  & $3.94\mathrm{e}{-6}$          & $3.29\mathrm{e}{-5}$          & $8.16\mathrm{e}{-8}$          & 0.04                        & 0.41                          & 0.00                          \\
			& 500                  & $2.99\mathrm{e}{-5}$          & $1.51\mathrm{e}{-4}$          & $4.85\mathrm{e}{-7}$          & 0.18                        & 1.23                          & 0.00                          \\
			& 1000                 & $2.3\mathrm{e}{-4}$           & $8.35\mathrm{e}{-4}$          & $3.89\mathrm{e}{-6}$          & 0.72                        & 3.48                          & 0.01                          \\ \hline\noalign{\smallskip}
		\end{tabular}
		\label{table:inverseProblem}
	}
\end{table}
In the second part of this section, we present the performance of the IR-SMD method on a text classification problem. For this experiment, in \eqref{eqn:problem3} we let $\mathcal{L}$ to be a hinge loss function, i.e., $\mathcal{L}(\langle x,a\rangle, b) \triangleq \max \{0, 1-b \langle x,a\rangle\}$ and the pair $(a,b)$ be generated using the dataset Reuters Corpus Volume I (RCV1) (see \cite{Lewis04}). This dataset is a collection of articles produced by Reuters between 1996 and 1997 that are categorized into four groups including Corporate/Industrial, Economics, Government/Social and Markets. We consider binary classification of articles with respect to the Markets class. After the tokenization process, each article is represented by a sparse binary vector, where 1 denotes existence and 0 denotes nonexistence of a token in the article. We use a subset of the data with $150,000$ articles and $138,921$ tokens. To induce sparsity of the optimal solution, we consider the problem \eqref{def:SL}
\begin{align} \label{numeric upper2}
\mbox{minimize}\ \ &h(x)\triangleq  \frac{\mu_h }{2}\|x\|_2^2+ \|x\|_1\\
\hbox{subject to}\ \ &x \in \argmin_{y \in \mathbf R^n } \EXP{ \mathcal{L}(\langle y,a\rangle, b)}, \nonumber
\end{align}
where we let $\mu_h=0.1$. Figure \ref{fig:rcv} shows the simulation results. Here we vary the initial stepsize, initial regularization parameter, initial vector $x_0$, and the averaging parameter $r<1$. We let $\g_k$ and $\lambda_k$ be given by the rules in Theorem \ref{thm rate}, and report the logarithm of averaged of the loss function $\mathcal{L}$ using 15 sample paths of size $10,000$. The plots in Figure \ref{fig:rcv} support convergence of the IR-SMD method for the optimization problem \eqref{numeric upper2}.
\begin{table}[]
	\centering
	\scalebox{0.9}{
		\begin{tabular}{l |  c c c}
			$(\gamma_0,\lambda_0)$ & $x_0=-10\times\mathbf{1}_n$ &  $x_0=\mathbf{0}_n$  &  $x_0=10\times\mathbf{1}_n$  \\ \hline\\
			(10,1)
			&
			\begin{minipage}{.3\textwidth}
				\includegraphics[scale=.26, angle=0]{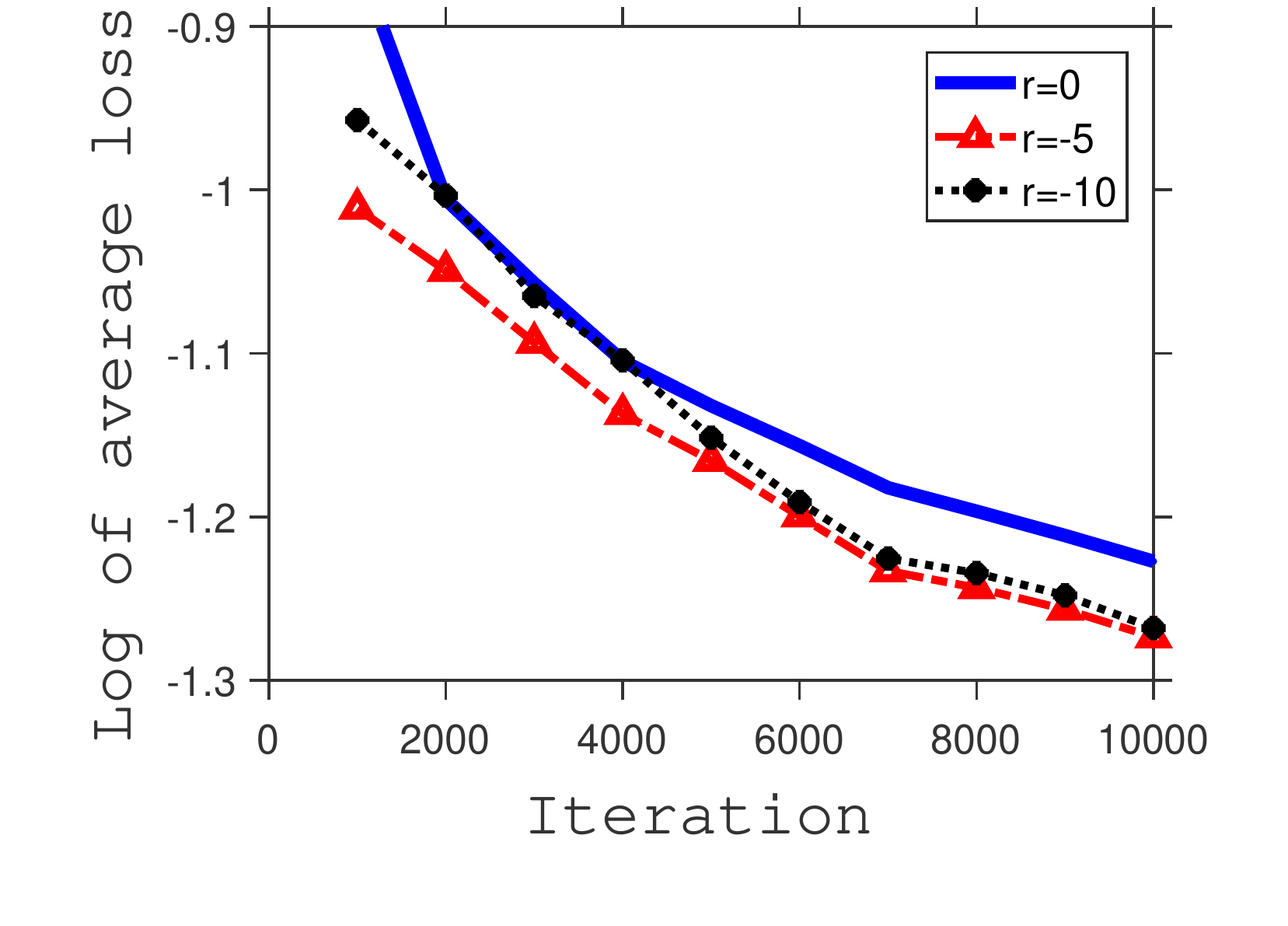}
			\end{minipage}
			&
			\begin{minipage}{.28\textwidth}
				\includegraphics[scale=0.26, angle=0]{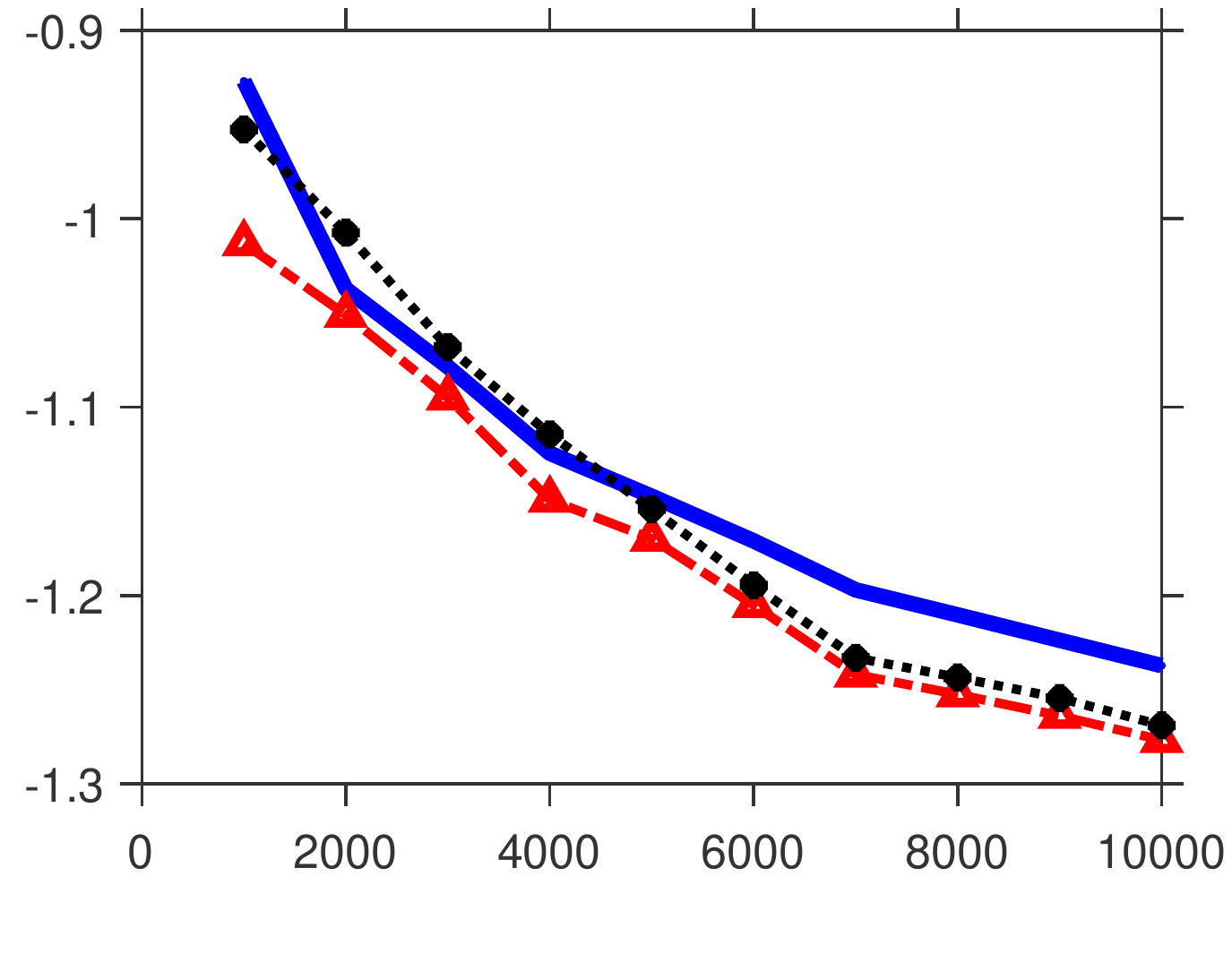}
			\end{minipage}
			&
			\begin{minipage}{.28\textwidth}
				\includegraphics[scale=.26, angle=0]{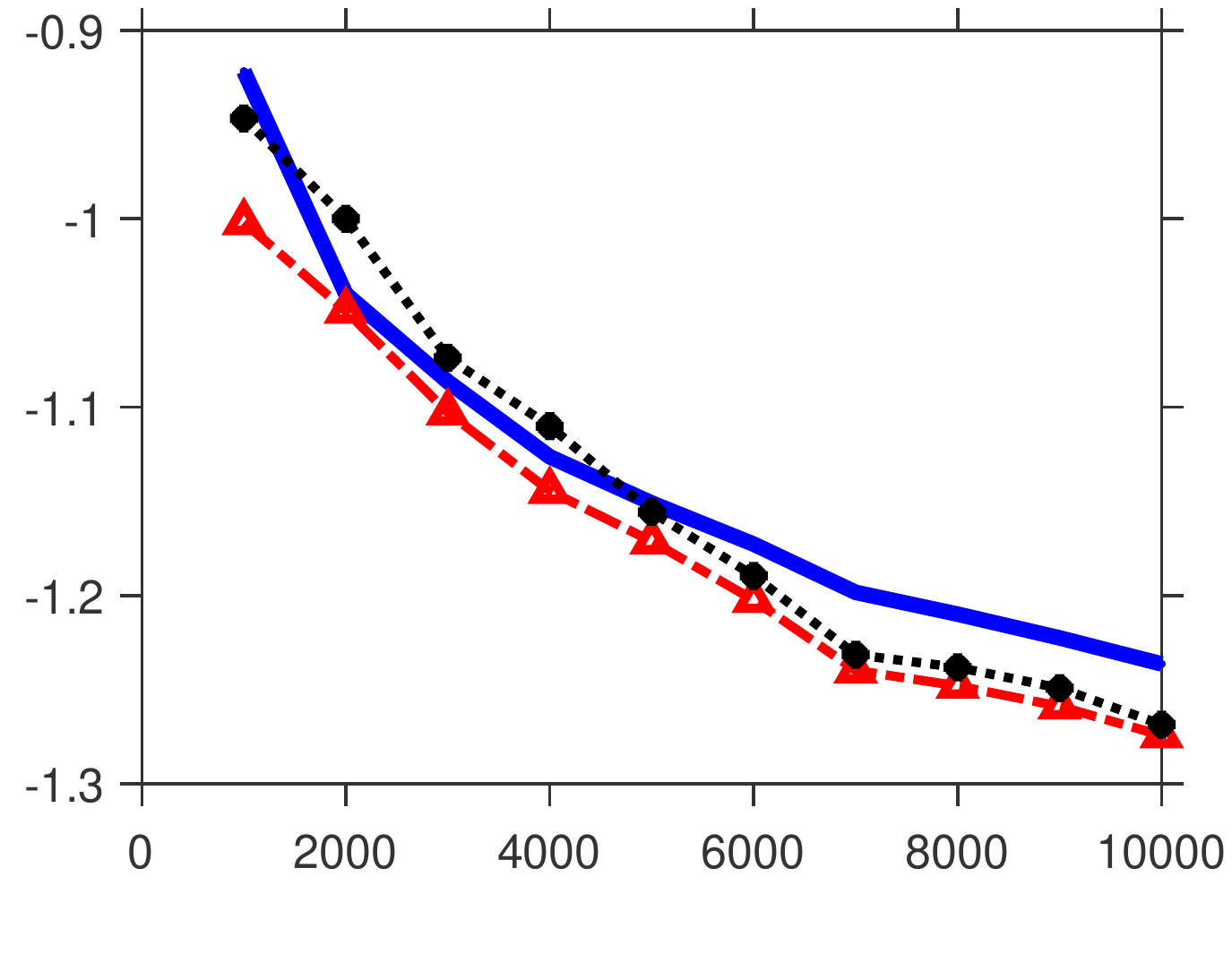}
			\end{minipage}
			\\
			(1,10)
			&
			\begin{minipage}{.3\textwidth}
				\includegraphics[scale=.26, angle=0]{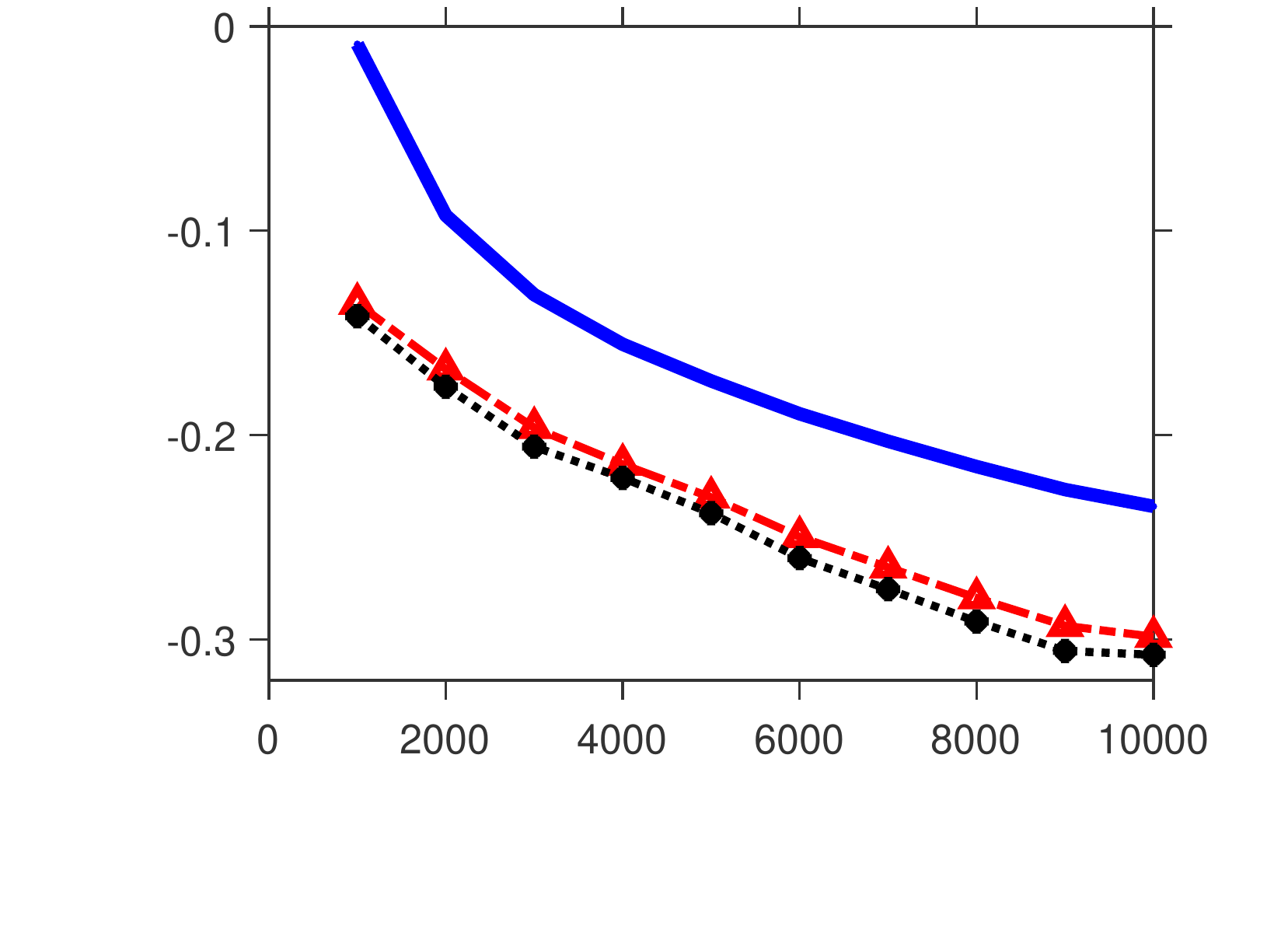}
			\end{minipage}
			&
			\begin{minipage}{.28\textwidth}
				\includegraphics[scale=.26, angle=0]{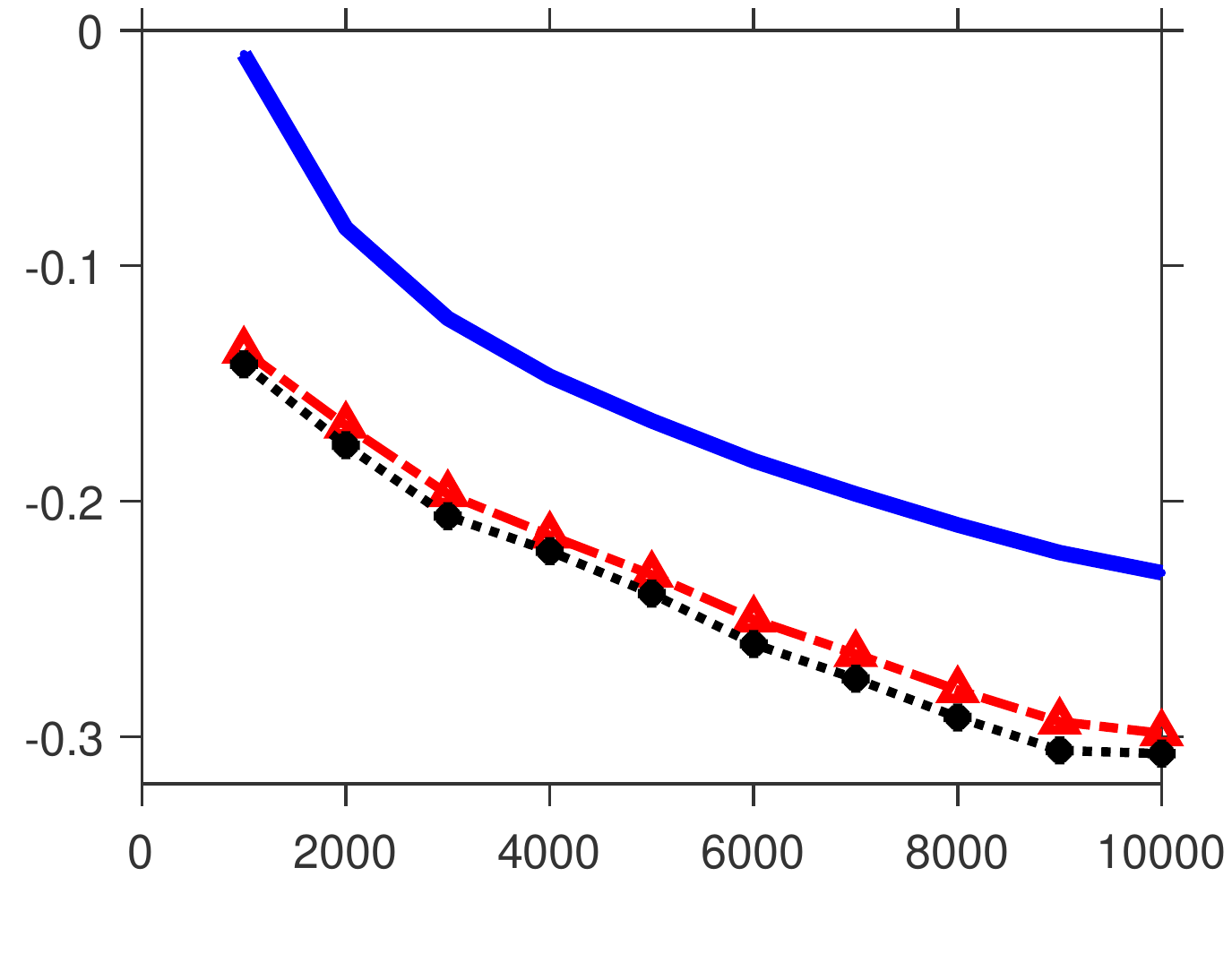}
			\end{minipage}
			&
			\begin{minipage}{.28\textwidth}
				\includegraphics[scale=.26, angle=0]{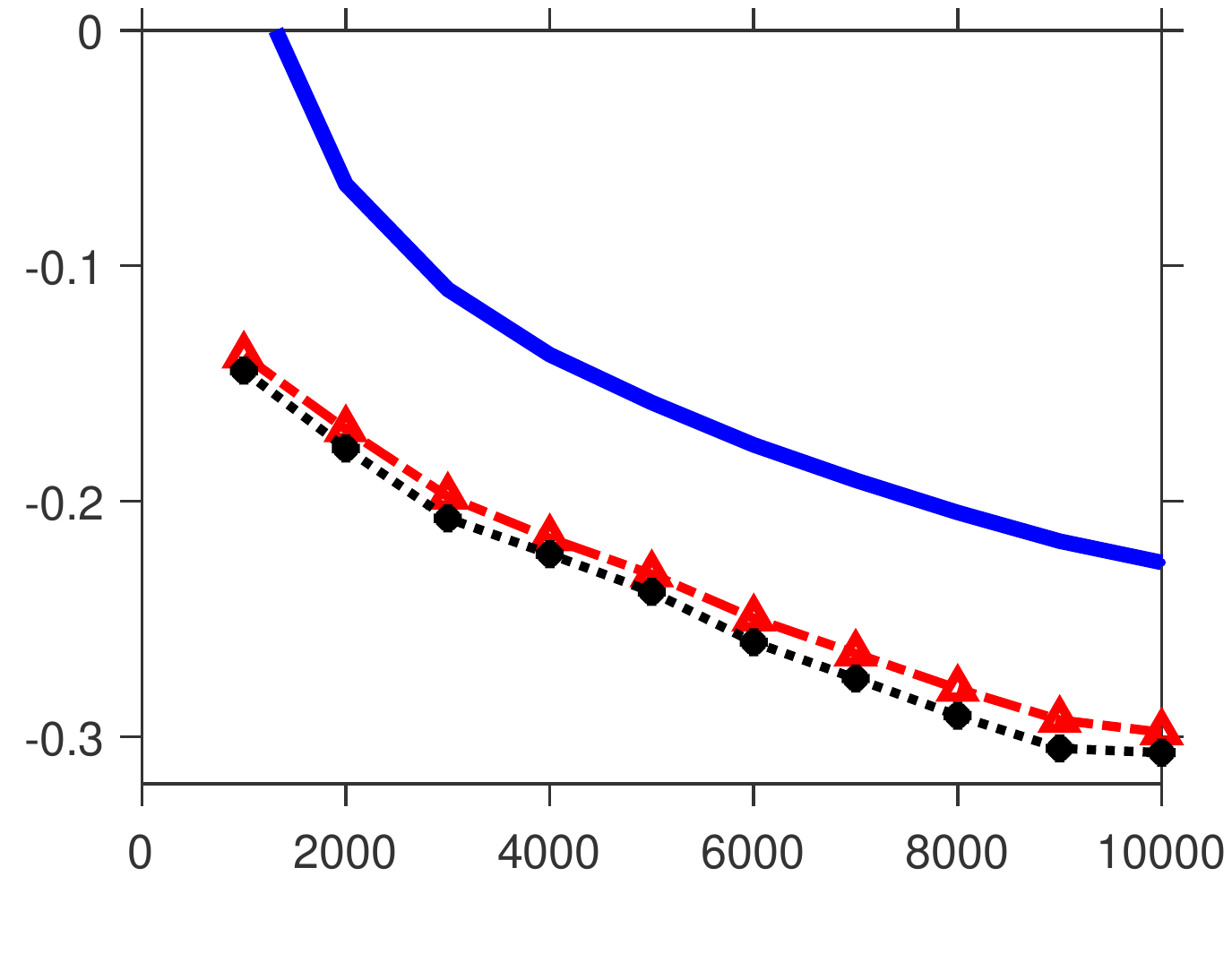}
			\end{minipage}
			\\
			(0.1,0.1) \  
			&
			\begin{minipage}{.3\textwidth}
				\includegraphics[scale=.26, angle=0]{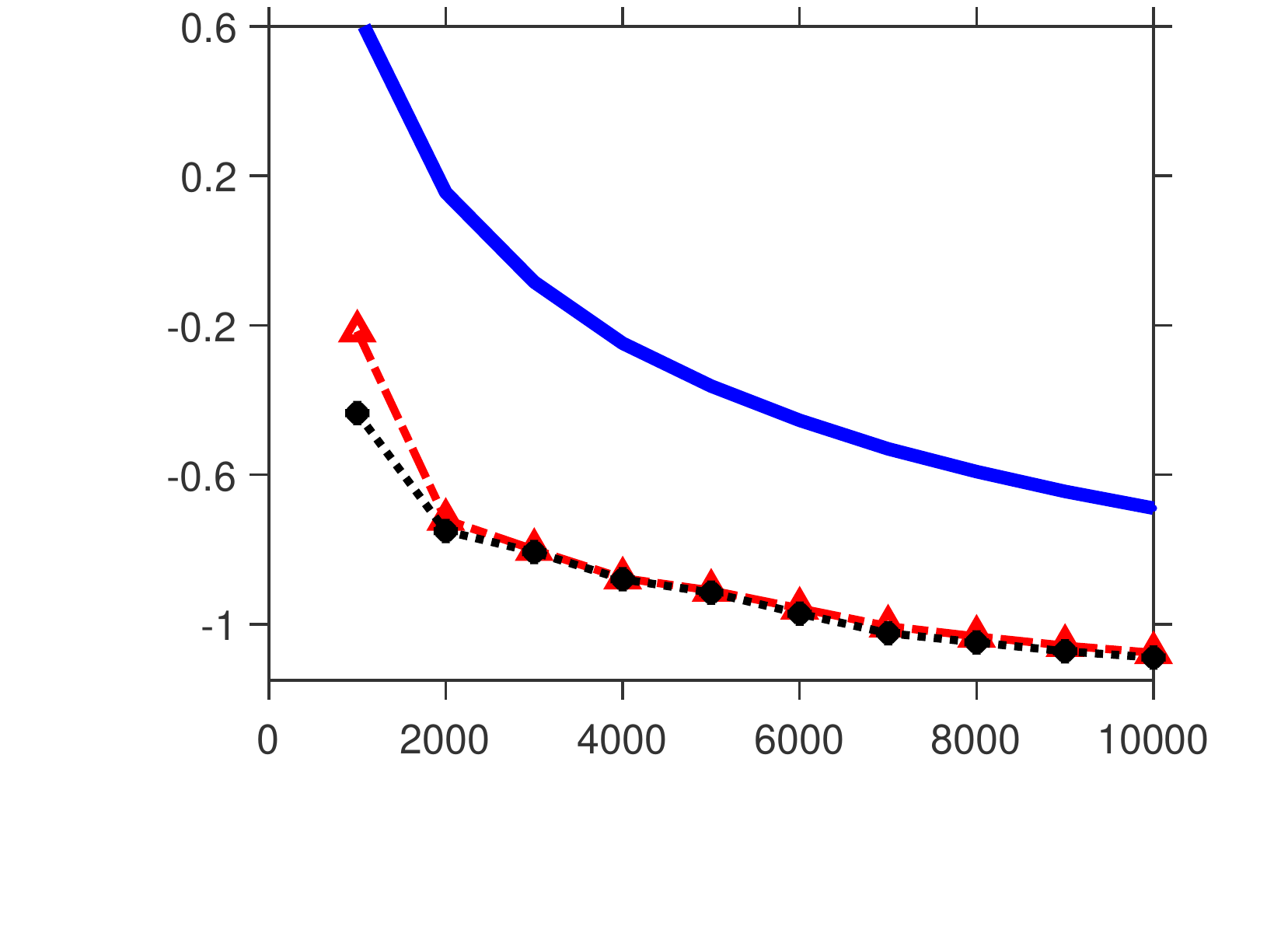}
			\end{minipage}
			&
			\begin{minipage}{.28\textwidth}
				\includegraphics[scale=.26, angle=0]{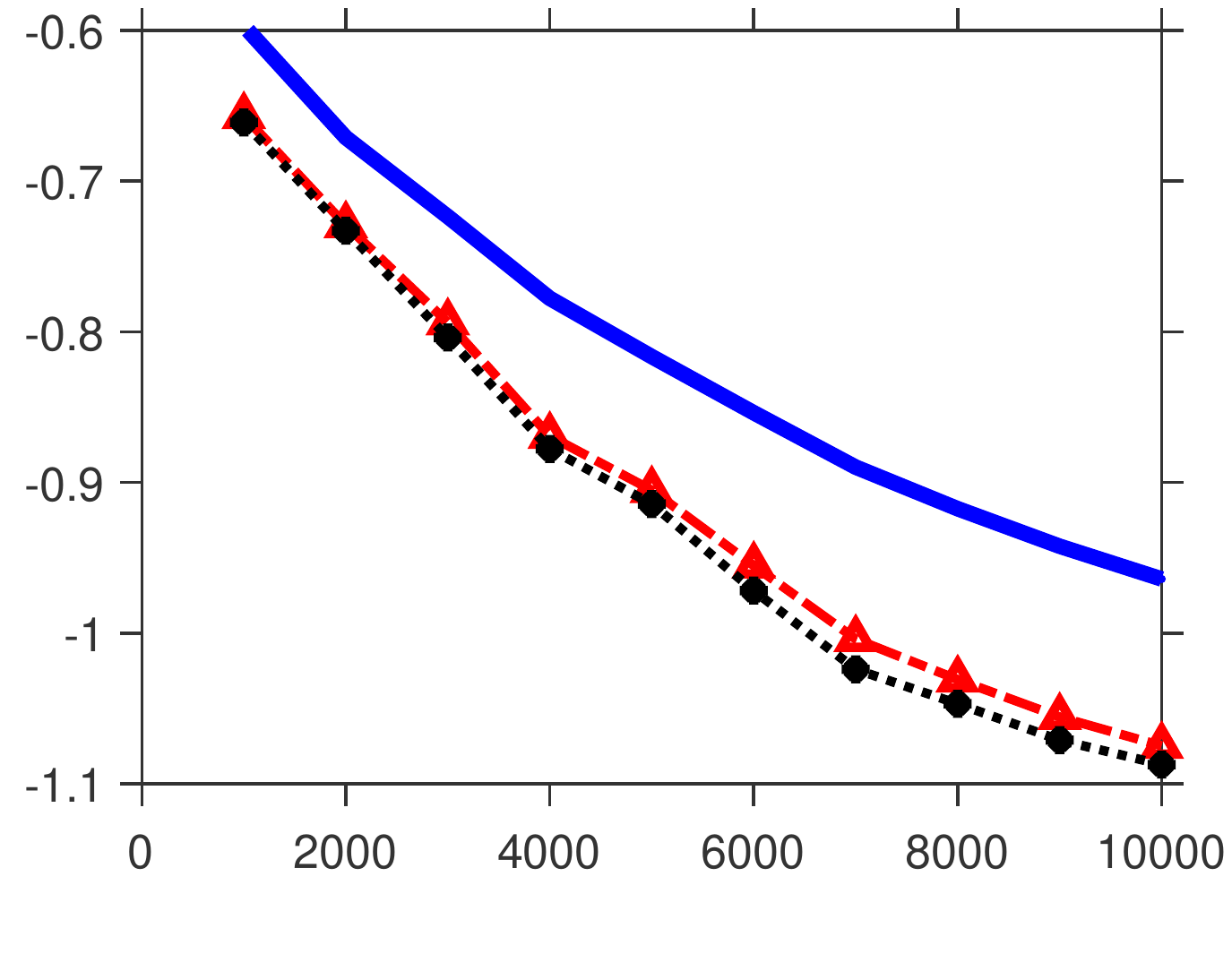}
			\end{minipage}
			&
			\begin{minipage}{.28\textwidth}
				\includegraphics[scale=.26, angle=0]{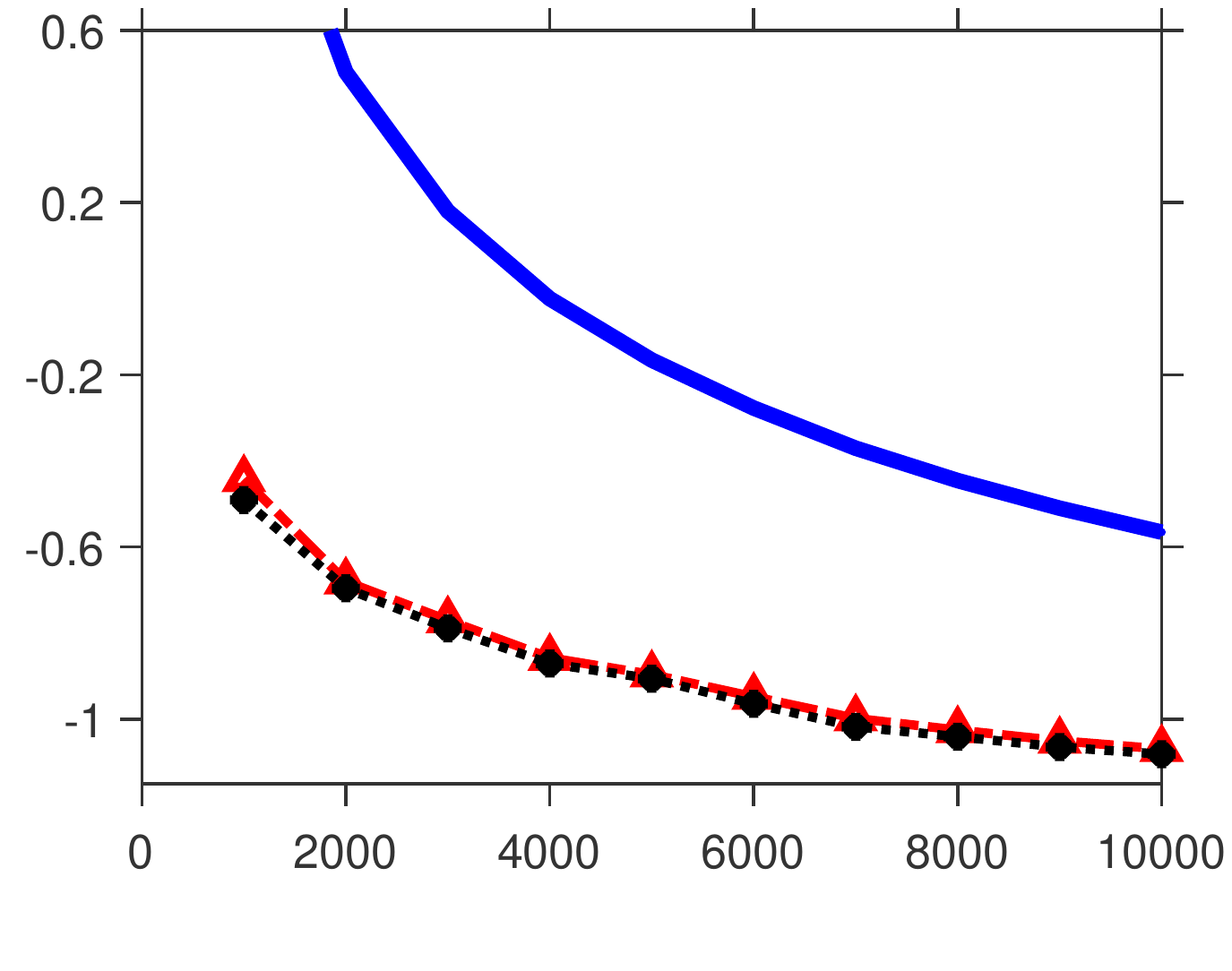}
			\end{minipage}
	\end{tabular}}
	\captionof{figure}{Performance of IR-SMD method for RCV1 dataset}
	\label{fig:rcv}
\end{table}
\section{Conclusions}\label{sec:rem}
Motivated by applications arising from machine learning and signal processing, we consider problem \eqref{def:SL}, where the optimal solution set of the problem \eqref{def:firstlevel} is the feasible set. We assume the objective function in \eqref{def:firstlevel} is convex and in the \eqref{def:SL} is strongly convex. We let both functions to be nondifferentiable and each one be given as an expected value of a stochastic function. We develop an iterative regularized stochastic mirror descent method (IR-SMD) where at each iteration, both stepsize and regularization parameters are updated. Our main contribution is two-fold: (i) We derive a family of update rules for the stepsize and the regularization parameter. We show that under this class of update rules, the sequence generated by the IR-SMD method converges to the unique optimal solution of problem \eqref{def:SL} in both an almost sure and a mean sense; (ii) To provide convergence rate statements, we show that under specific update rules for the stepsize and regularization sequences, the expected feasibility gap converges to zero with a near optimal rate of convergence. Moreover, under these update rules, the almost sure convergence and convergence in mean to the optimal solution of the bilevel problem are guaranteed. Our preliminary numerical results performed on three types of linear inverse problems and a binary text classification problem are promising.  

\bibliographystyle{amsplain}
\bibliography{reference1}
\appendix
\section{Proofs}\label{Appendix}
\subsection{\textbf{Proof of Lemma \ref{lemma 2-stage 2}}} \label{proof of lemma 2-stage 2}
\begin{proof}
	We prove this lemma in two steps.\\
	\noindent \textbf{Step1:} For any fixed $z\in Z$ we show that
	\begin{align} \label{minsum = summin}
	\min_{y\in Y(z)} \sum_{i=1}^{N} p_i q(y_i,\xi_i)= \sum_{i=1}^{N} p_i \min_{y_i \in Y_i(z)} q(y_i,\xi_i),
	\end{align}
	where the set $Y(z) \in \mathbf{R}^{m \times N}$ is defined as $Y(z)\triangleq \prod_{i=1}^N Y_i(z)$.
	For this, we denote the optimization problem on the left of the equation \eqref{minsum = summin} by $P_{lhs}$, and the minimization problem $\min_{y_i \in Y_i(z)} q(y_i,\xi_i)$ on the right hand side by $P_{rhs}^i$ for all $i=1,\cdots, N$. To prove the desired statement, we show that a vector ${y^*}^T\triangleq \left({y_1^*}^T,\cdots, {y_N^*}^T\right) \in \mathbf{R}^{m \times N}$ is an optimal solution to problem $P_{lhs}$ if and only if for all $i=1,\cdots, N$, $y_i^*$ is an optimal solution to problem $P_{rhs}^i$. First, let us assume $y^*$ solves $P_{lhs}$. Then, from the optimality condition we can write
	\begin{align} \label{optimality condition1}
	\left( p_1\nabla^Tq(y_1^*,\xi_1), \cdots , p_N\nabla^Tq(y_N^*,\xi_N) \right)^T(y-y^*) \geq 0, \qquad \hbox{for all } y\in Y(z),
	\end{align}
	where we use the notation $y^T=\left(y_1^T, \cdots,y_N^T\right)$. Then, we obtain
	\begin{align} \label{optimality condition2}
	\sum_{i=1}^{N} p_i \nabla^Tq(y_i^*,\xi_i) (y_i-y_i^*) \geq 0, \qquad \hbox{for all } y\in Y(z).
	\end{align}
	Consider any arbitrary $i \in \{1,\cdots, N\}$ and an arbitrary $y_i \in Y_i(z)$. Let us define ${\hat y(i)}^T=\left({y_1^*}^T,\cdots,{y_{i-1}^*}^T,{y_i}^T,{y_{i+1}^*}^T, \cdots,{y_N^*}^T\right)$. Clearly $\hat y(i) \in Y(z)$ and from \eqref{optimality condition1} we have
	\begin{align}
	0 \leq \left( p_1 \nabla^Tq(y_1^*,\xi_1), \cdots , p_N \nabla^T q(y_N^*,\xi_N) \right)^T(\hat y(i)-y^*)=p_i \nabla^T q(y_i^*,\xi_i) (y_i-y_i^*),
	\end{align}
	implying that $y_i^*$ is an optimal solutoin to $P_{rhs}^i$ (because $p_i\geq 0$). Since this holds for any $i=1,\cdots, N$, we conclude that if $y^*$ solves $P_{lhs}$, then $y_i^*$ solves problem $P_{rhs}$ for $i=1,\cdots, N$. 
	To show the converse, assume $y_i^* \in Y_i(z)$ solves problem $P_{rhs}^i$ for all $i=1,\cdots,K$. By optimality condition we have
	\begin{align}
	\nabla^T q(y_i^*,\xi_i) (y_i-y_i^*) \geq 0, \qquad \hbox{for all } y_i\in Y_i(z).
	\end{align}
	Multiplying both sides of the preceding inequality by $p_i$ and summing  over $i$, we can see that \eqref{optimality condition2} holds, suggesting that $y^*$ solves problem $P_{lhs}$.
	
	\noindent \textbf{Step2:} We consider the following problem
	\begin{align}\label{minmin} 
	\min_{z \in Z} \min_{y \in Y(z)} c(z)+q(y,\xi),
	\end{align}
	where $Y(z)$ and $y$ are defined in Step1 and we denote $q(y,\xi) \triangleq \sum_{i=1}^{N} p_i q(y_i,\xi_i)$. We show that the optimal objective value of \eqref{minmin} is equal to the optimal objective value of the following problm 
	\begin{align} \label{minmin2} 
	\min_{z\in Z, y \in Y(z)} c(z)+q(y,\xi).
	\end{align}
	For the problem $\min_{y \in Y(z)} c(z)+q(y,\xi)$, let us define the Lagrangian function 
	\begin{align*}
	L(z,y,\nu)\triangleq c(z)+q(y,\xi)+ \sum_{i=1}^{N} \sum_{j=1}^{J} \nu _{ij}(t_j(z)+w_j(y_i,\xi_i)), \ \hbox{for } z \in Z, y_i \in Y,\ \nu \in \mathbf{R}_+^{I\times J}.
	\end{align*}
	First, let us assume $z\in Z$ is arbitrary and fixed. Taking supremum with respect to $\nu$ and considering the definition of $Y(z)$ we can write
	\begin{align*}
	\sup_{\nu} L(z,y,\nu)=&\sup_{\nu} \left\{c(z)+q(y,\xi)+ \sum_{i=1}^{N} \sum_{j=1}^{J} \nu _{ij}(t_j(z)+w_j(y_i,\xi_i)) \right\}\\=&\begin{cases}
	c(z)+q(y,\xi) &\mbox{if } y \in Y(z) , \vspace{0.3cm}\\
	+\infty &\hbox{otherwise}.
	\end{cases}
	\end{align*}
	Taking infimum with respect to $y$ and taking into account that $Y(z) \subseteq Y^N$ we have
	\begin{align*}
	\inf_{y \in  Y} \sup_{\nu} L(z,y,\nu)=\inf_{y \in Y(z)} \{c(z)+q(y,\xi)\}.
	\end{align*}
	Now, by taking infimum with respect to $z \in Z$, we obtain
	\begin{align}\label{infinf}
	\inf_{z \in Z} \inf_{y \in  Y} \sup_{\nu} L(z,y,\nu)=\inf_{z \in Z} \inf_{y \in Y(z)} \{c(z)+q(y,\xi)\}.
	\end{align}
	Second, 
	for any $z \in Z$, we have 
	\begin{align*}
	\sup_{\nu} L(z,y,\nu)&=\sup_{\nu} \left\{c(z)+q(y,\xi)+ \sum_{i=1}^{N} \sum_{j=1}^{J} \nu _{ij}(t_j(z)+w_j(y_i,\xi_i)) \right\}\\&=\begin{cases}
	c(z)+q(y,\xi) &\mbox{if } \sum_{i=1}^{N} \sum_{j=1}^{J} \nu _{ij}(t_j(z)+w_j(y_i,\xi_i)) \leq 0 , \vspace{0.3cm}\\
	+\infty &\hbox{otherwise}.
	\end{cases}
	\end{align*}
	Thus, we have 
	\begin{align*}
	\inf_{z \in Z, y \in  Y} \sup_{\nu}  L(z,y,\nu)=\inf_{z \in Z, y \in  Y,  \sum_{i=1}^{N} \sum_{j=1}^{J} \nu _{ij}(t_j(z)+w_j(y_i,\xi_i)) \leq 0}  \{ c(z)+q(y,\xi)\}.
	\end{align*}
	Considering the definition of $Y(z)$, we can write
	\begin{align}\label{infinf2}
	\inf_{z \in Z} \inf_{y \in  Y} \sup_{\nu } L(z,y,\nu)=\inf_{z \in Z, y \in Y(z)} \{c(z)+q(y,\xi)\}.
	\end{align}
	Equations \eqref{infinf} and \eqref{infinf2} imply that the following holds
	\begin{align*}
	\inf_{z \in Z} \inf_{y \in Y(z)} \{c(z)+q(y,\xi)\}=\inf_{z \in Z, y \in Y(z)} \{c(z)+q(y,\xi)\}.
	\end{align*}
	
	Now it can be easily seen  that by combining Step1 and Step2, model \eqref{single stage} can be writen as \eqref{compact two-stage}.
\end{proof}

\subsection{\textbf{Proof of Lemma \ref{two-stage2bilevel}}} \label{proof of lemma two-stage2bilevel}
\begin{proof}
	Since $\xi$ has a distribution with a finite support, we have
	\begin{align}
	\EXP{Q(z,\xi)}=\sum_{i=1}^{N} p_i Q(z,\xi_i).
	\end{align}
	Replacing this in \eqref{single stage}, and taking into account that  all the assumptions for applying Lemma \ref{lemma 2-stage 2} are met, we can write problem \eqref{single stage} in the form of \eqref{compact two-stage}.
	Let $H$ and $F$ be as given in the statement of the lemma.
	Now, we show that $x^*\in X$ solves problem  \eqref{compact two-stage} if and only if it solves \eqref{bilevel two-stage2}.\\
	$(\Rightarrow)$ Assume  $x^*\in X$ solves problem  \eqref{compact two-stage}. Therefore, $x^*$ satisfies all the constraints of problem \eqref{compact two-stage}. Considering the definition of $F$, it is easy to see that $\EXP{F(x^*,\xi)}=0$. In addition, note that the definition of $F$ implies that $\EXP{F(x^*,\xi)}\geq 0$. Thus, we conclude that $x^* \in \argmin_{x \in X } \EXP{F(x,\xi)}$. This implies that $x^*$ is a feasible solution to problem \eqref{bilevel two-stage2}. To show optimality of $x^*$ for \eqref{bilevel two-stage2}, assume there is a feasible solution of problem \eqref{bilevel two-stage2}, $\hat x\neq x^*$ at which  $h(\hat x)< h(x^*)$, where $h(x)\triangleq \EXP{H(x,\xi)}$. Note that by feasibility of $\hat x$, we have $\hat x \in \argmin_{x \in X } \EXP{F(x,\xi)}$. We already know that $\min_{x \in X } \EXP{F(x,\xi)}$ is achieved at zero by $x^*$. Therefore, we have $\EXP{F(\hat x,\xi)}=0$. Using this and taking into account the definition of $F$ again, we have that $\hat x$ is a feasible solution to problem \eqref{compact two-stage}. Taking to account that problems \eqref{compact two-stage} and \eqref{bilevel two-stage2} have identical objective functions $h$, and that $h(\hat x)< h(x^*)$, the optimality of $x^*$ from problem \eqref{compact two-stage} is contradicted. As such, we conclude that if $x^*\in X$ solves problem \eqref{compact two-stage}, then it solves problem \eqref{bilevel two-stage2} as well.

	$(\Leftarrow)$  Now assume that $x^*$ solves problem \eqref{bilevel two-stage2}. Note that from the assumptions of Lemma \ref{lemma 2-stage 2}, problem \eqref{compact two-stage} is feasible. Therefore, there exists $x_0\in X$ that satisfies all the constraints of problem \eqref{compact two-stage}. This means that $\min_{x \in X } \EXP{F(x,\xi)}=0$. So, $\EXP{F(x^*,\xi)}=0$ by its feasibility for \eqref{bilevel two-stage2}. Taking this into account and using the definition of $F$, we can conclude that $x^*$ is a feasible solution for problem \eqref{compact two-stage}. Now, considering that problems \eqref{compact two-stage} and \eqref{bilevel two-stage2} have identical objective functions $h$, we conclude that $x^*$ is also an optimal solution for problem \eqref{compact two-stage}.
\end{proof}

\subsection{\textbf{Proof of Lemma \ref{lem:unique sol for h}}}\label{proof of lem:unique sol for h}
\begin{proof}
	\begin{itemize}
		\item[(a)] By Assumption \ref{assum:properties}(b,c), from convexity of $f$ and strong convexity of $h$, for all $x,y \in X$ we have
		\begin{align*}
		&f(x)+\langle g_f(x), y-x \rangle \leq f(y),\\
		&h(x)+\langle g_h(x), y-x \rangle + \frac{\mu_h}{2}\|x-y\|^2 \leq h(y).
		\end{align*}
		Multiplying the second inequality by the nonnegative parameter $\lambda$ and adding it to the first one, we will obtain strong convexity with parameter $\mu_h$ for the objective function of problem \eqref{regularized form}. Since by Assumption \ref{assum:properties}(a), set $X$ is compact and convex, the uniqueness will follow from subdifferentiability of $f$ in a similar fashion to the proof of Theorem $2.2.6$ of~\cite{Nesterov04}, page 85.
		\item[(b)] From Assumption \ref{assum:properties}(c), $h$ is strongly convex. We need to prove the set $X^*$, on which problem \eqref{def:SL} is defined, is compact and convex. We can write: $X^*=X \cap \{x\in \mathbf{R}^n | f(x)\leq f^*\}$.
		So, $X^*$ is the intersection of two compact and convex sets; i.e., $X$ and the $f^*$ sublevel set of a convex function, $f$. So $X^*$ is also compact and convex. The rest of proof follows from Theorem $2.2.6$ of~\cite{Nesterov04}, page 85.
	\end{itemize}
\end{proof}
\subsection{\textbf{Proof of Lemma \ref{results from convexity of f and strong convexity of h}}} \label{proof of results from convexity of f and strong convexity of h}
\begin{proof} { To show relation \eqref{result from convexity of f}, from convexity of $f$ we have
		\begin{align*}
		f(x_{\lambda_k}^*)+\langle g_f(x_{\lambda_k}^*),x-x_{\lambda_k}^* \rangle \leq f(x) \qquad \hbox{for all } x \in X, \hbox{where}\ g_f(x_{\lambda_k}^*) \in \partial f(x_{\lambda_k}^*),
		\end{align*}
		Similarly, we can write
		\begin{align*}
		f(x_{\lambda_{k-1}}^*)+\langle g_f(x_{\lambda_{k-1}}^*),y-x_{\lambda_{k-1}}^* \rangle \leq f(y) \qquad \hbox{for all } y \in X,\end{align*}
		where $g_f(x_{\lambda_{k-1}}^*) \in \partial f(x_{\lambda_{k-1}}^*)$.
		
		Let $x:=x_{\lambda_{k-1}}^*$ and $y:= x_{\lambda_k}^*$ in the preceding inequalities. Then, relation \eqref{result from convexity of f} is obtained by adding the resulting two inequalities.\\
		Nest, we show relation \eqref{result from strong convexity of h}. From strong convexity of $h$, we can write 
		\begin{align*}
		h(x_{\lambda_k}^*)+\langle g_h(x_{\lambda_k}^*),x-x_{\lambda_k}^* \rangle +\frac{\mu_h}{2} \|x_{\lambda_k}^*-x\|^2 &\leq h(x) \qquad \hbox{for all } x \in X, \end{align*} 
		where $g_h(x_{\lambda_k}^*) \in \partial h(x_{\lambda_k}^*)$,\\ 
		\begin{align*}
		h(x_{\lambda_{k-1}}^*)+\langle g_h(x_{\lambda_{k-1}}^*),y-x_{\lambda_k}^* \rangle +\frac{\mu_h}{2} \|x_{\lambda_{k-1}}^*-y\|^2 &\leq h(y) \qquad \hbox{for all } y \in X, \end{align*}
		where $g_h(x_{\lambda_{k-1}}^*) \in \partial h(x_{\lambda_{k-1}}^*)$.
		Now we let $x:=x_{\lambda_{k-1}}^*$ and $y:= x_{\lambda_k}^*$ in these inequalities. Relation \eqref{result from strong convexity of h} will be obtained by adding the resulting relations.
	}
\end{proof}

\subsection{\textbf{Proof of Proposition \ref{prop:xk_estimate}}} \label{proof of prop:xk_estimate}
\begin{proof}{
		\begin{itemize}
			\item[(a)] Let $k \geq 1$ be fixed. Since {$x_{\lambda_k}^*$ is the minimizer of problem \eqref{regularized form} at $\lambda=\lambda_k$. from optimality conditions we have}
			\begin{align*}
			\langle g_f(x_{\lambda_k}^*) + \lambda_k g_h(x_{\lambda_k}^*),x-x_{\lambda_k}^* \rangle \geq 0 \qquad  \hbox{for all } x \in X.
			\end{align*}
			Similarly, from the optimality conditions of problem \eqref{regularized form} at $\lambda=\lambda_{k-1}$, we can write
			\begin{align*}
			\langle g_f(x_{\lambda_{k-1}}^*) + \lambda_{k-1}g_h(x_{\lambda_{k-1}}^*),y-x_{\lambda_{k-1}}^* \rangle \geq 0 \qquad \hbox{for all } y \in X.
			\end{align*}
			Let $x:=x_{\lambda_{k-1}}^*$ and $y:= x_{\lambda_k}^*$ in the preceding two inequalities. By adding these relations we obtain
			\begin{align*}
			\langle g_f(x_{\lambda_k}^*) + \lambda_k g_h(x_{\lambda_k}^*),x_{\lambda_{k-1}}^*-x_{\lambda_k}^* \rangle + \langle g_f(x_{\lambda_{k-1}}^*) + \lambda_{k-1}g_h(x_{\lambda_{k-1}}^*),x_{\lambda_k}^*-x_{\lambda_{k-1}}^* \rangle \geq 0.
			\end{align*}
			Therefore, by rearranging the left-hand side we obtain
			\begin{align}\label{proof1.1}
			\langle g_f(x_{\lambda_k}^*) - g_f(x_{\lambda_{k-1}}^*) ,x_{\lambda_{k-1}}^*-x_{\lambda_k}^* \rangle + \langle \lambda_k g_h(x_{\lambda_k}^*) - \lambda_{k-1}g_h(x_{\lambda_{k-1}}^*),x_{\lambda_{k-1}}^*-x_{\lambda_k}^* \rangle \geq 0.
			\end{align}
			Note that from relation \eqref{result from convexity of f} in Lemma \ref{results from convexity of f and strong convexity of h}, we have \\ $\langle g_f(x_{\lambda_k}^*) - g_f(x_{\lambda_{k-1}}^*) ,x_{\lambda_{k-1}}^*-x_{\lambda_k}^* \rangle \leq 0$. Thus, from \eqref{proof1.1} we obtain
			\begin{align}\label{proof1.2}
			\langle \lambda_k g_h(x_{\lambda_k}^*) - \lambda_{k-1}g_h(x_{\lambda_{k-1}}^*),x_{\lambda_{k-1}}^*-x_{\lambda_k}^* \rangle \geq 0.
			\end{align}
			Adding and subtracting $\langle \lambda_k g_h(x_{\lambda_{k-1}}^*),x_{\lambda_{k-1}}^*-x_{\lambda_k}^* \rangle$, it follows that
			\begin{align*}
			& \langle \lambda_k g_h(x_{\lambda_k}^*) - \lambda_k g_h(x_{\lambda_{k-1}}^*),x_{\lambda_{k-1}}^*-x_{\lambda_k}^* \rangle \\&+ 	\langle \lambda_k g_h(x_{\lambda_{k-1}}^*) - \lambda_{k-1} g_h(x_{\lambda_{k-1}}^*),x_{\lambda_{k-1}}^*-x_{\lambda_k}^* \rangle \geq 0.
			\end{align*}
			Therefore, by rearranging the terms we have
			\begin{align*}
			(\lambda_k- \lambda_{k-1}) \langle g_h(x_{\lambda_{k-1}}^*),x_{\lambda_{k-1}}^*-x_{\lambda_k}^* \rangle \geq \lambda_k\langle g_h(x_{\lambda_{k-1}}^*) - g_h(x_{\lambda_k}^*),x_{\lambda_{k-1}}^*-x_{\lambda_k}^* \rangle.
			\end{align*}
			Combining the preceding inequality with \eqref{result from strong convexity of h}, we obtain
			\begin{align*}
			(\lambda_k- \lambda_{k-1}) \langle g_h(x_{\lambda_{k-1}}^*),x_{\lambda_{k-1}}^*-x_{\lambda_k}^* \rangle \geq \mu_h \lambda_k \|x_{\lambda_k}^* - x_{\lambda_{k-1}}^*\|^2.
			\end{align*}
			By definition of dual norm $\|\cdot\|_*$, since $\|a\|_* = \sup_{\|b\| \leq1}{\langle a,b\rangle} $, we have $\|a\|_* \geq \langle a,b\rangle$ for $\|b\| \leq1$, so
			\begin{align*}
			|\lambda_k- \lambda_{k-1}|\|g_h(x_{\lambda_{k-1}}^*)\|_*\|x_{\lambda_{k-1}}^*-x_{\lambda_k}^*\| \geq \mu_h \lambda_k \|x_{\lambda_k}^* - x_{\lambda_{k-1}}^*\|^2.
			\end{align*}
			From Assumption \ref{assum:properties}(e) and Remark \ref{assumptionandJensen}, we have $\|g_h(x)\|_* \leq C_H$ for all $g_h \in \partial h(x)$ and $x \in X$. Thus,
			\begin{align*}
			|\lambda_k- \lambda_{k-1}|C_H\|x_{\lambda_{k-1}}^*-x_{\lambda_k}^*\| \geq \mu_h \lambda_k \|x_{\lambda_k}^* - x_{\lambda_{k-1}}^*\|^2.
			\end{align*}
			Let us assume $ x_{\lambda_k}^* \neq x_{\lambda_{k-1}}^*$. Then
			\begin{align*}\|x_{\lambda_k}^*-x_{\lambda_{k-1}}^*\|\leq \frac{C_H}{\mu_h}\left	|1-\frac{\lambda_{k-1}}{\lambda_k} \right|.
			\end{align*}
			If $x_{\lambda_k}^* = x_{\lambda_{k-1}}^*$, then $\|x_{\lambda_k}^*-x_{\lambda_{k-1}}^*\|=0 \leq \frac{C_H}{\mu_h}\left	|1-\frac{\lambda_{k-1}}{\lambda_k} \right|$. Therefore the desired inequality holds.
			\item[(b)] Let $x^*$ be the minimizer of function $f$ over the set $X$ and $x_{\lambda_k}^*$ be the minimizer of \eqref{regularized form} at $\lambda=\lambda_k$. From the optimality conditions for this problem we have
			\begin{align*}
			\langle g_f(x_{\lambda_k}^*) + \lambda_k g_h(x_{\lambda_k}^*),x-x_{\lambda_k}^* \rangle \geq 0 \qquad \hbox{for all } x \in X.
			\end{align*}
			Similarly, we can write for any arbitrary $x^* \in X^*$
			\begin{align*}
			\langle g_f(x^*),y-x^* \rangle \geq 0 \qquad \hbox{for all } y \in X.
			\end{align*}
			Let $x:=x^*$ and $y:= x_{\lambda_k}^*$ in the preceding inequalities. Then by adding them we obtain
			\begin{align*}
			\langle g_f(x_{\lambda_k}^*) + \lambda_k g_h(x_{\lambda_k}^*)-g_f(x^*),x^*-x_{\lambda_k}^* \rangle \geq 0.
			\end{align*}
			Rearranging the inequality, we have
			\begin{align*}
			\langle \lambda_k g_h(x_{\lambda_k}^*),x^*-x_{\lambda_k}^* \rangle \geq \langle g_f(x^*) - g_f(x_{\lambda_k}^*),x^*-x_{\lambda_k}^* \rangle.
			\end{align*}
			By convexity of $f$, we know that $\langle g_f(x^*) - g_f(x_{\lambda_k}^*),x^*-x_{\lambda_k}^* \rangle \geq 0$. So from the preceding relation we obtain
			\begin{align} \label{proof b1}
			\langle g_h(x_{\lambda_k}^*),x^*-x_{\lambda_k}^* \rangle \geq 0.
			\end{align}
			By convexity of $h$, for all $x,y \in X$ we can also have 
			\begin{align*}
			h(x) \geq h(y)+ \langle g_h(y),x-y \rangle,
			\end{align*}
			where $g_h(y)\in \partial h(y)$. Now by letting $x:=x^*$ and $y:= x_{\lambda_k}^*$, we have
			\begin{align}\label{proof b2}
			h(x^*) \geq h(x_{\lambda_k}^*)+ \langle g_h(x_{\lambda_k}^*),x^*-x_{\lambda_k}^* \rangle.
			\end{align}
			Comparing \eqref{proof b1} and \eqref{proof b2}, we obtain that $	h(x^*) \geq h(x_{\lambda_k}^*)$. Note that $x_h^* \in X^*$ implying that $x_h^*$ is an optimal solution to problem \eqref{def:firstlevel}. Therefore, for $ x^* := x_h^*$ we obtain
			\begin{align}\label{proof b3}
			h(x_h^*) \geq h(x_{\lambda_k}^*) \qquad \hbox{for all } k \geq 0.
			\end{align}
			Now consider the sequence $\{x_{\lambda_k}^*\}$. We know $x_{\lambda_k}^* \in X$. Since $X$ is bounded, $\{x_{\lambda_k}^*\}$ has at least one accumulation point. Let $\hat{x}^*$ be an accumulation point of  $\{x_{\lambda_k}^*\}$. From optimality of $x_{\lambda_k}^*$ we have
			\begin{align*}
			f(x_{\lambda_k}^*) + \lambda_k h(x_{\lambda_k}^*)\leq f(x) + \lambda_k h(x) \qquad \hbox{for all } x \in X.  
			\end{align*}
			Taking limits along the convergent subsequence from both sides of the preceding inequality for all $x \in X$ and considering the assumption that $\lambda_k \rightarrow0$ and continuity of $f$ and $h$, we have
			\begin{align*}
			f(\hat{x}^*) \leq f(x) \qquad \hbox{for all } x \in X.  
			\end{align*}
			Thus, $\hat{x}^* \in X^*$ implying that any arbitrary accumulation point of $\{x_{\lambda_k}^*\}$ is an optimal solution to problem \eqref{def:firstlevel}.\\
			Now let $\{x_{\lambda_{k_i}}^*\}$ be an arbitrary convergent subsequence of $\{x_{\lambda_k}^*\}$ with accumulation point $\tilde x^*$. Taking limits of \eqref{proof b3} along $\{x_{\lambda_{k_i}}^*\}$ we have $	h(x_h^*) \geq h(\tilde x^*).$			
			But from Lemma \ref{lem:unique sol for h} we know that $x_h^*$ is the unique optimal solution of problem \eqref{def:SL}. Thus, $\tilde x^*=x_h^* $. Therefore, all the limit points of $\{x_{\lambda_k}^*\}$ converge to $x_h^*$. Hence, $\lim_{k \to \infty} {x_{\lambda_k}}$ exists and is equal to $x_h^*$. 
	\end{itemize} }
\end{proof}
\subsection{ \textbf{Proof of Lemma \ref{lemma general conv lemma}}} \label{proof of lemma general conv lemma}
\begin{proof}
	Let $e_k=\EXP{\nu_k}$ for all $k$. We prove this lemma by applying induction. At first we need to show that the result is true for $k=0$. By definition of $\tau$ we have $\tau \geq \frac{\EXP{\nu_1} \alpha_0}{\beta_0}$. So $\EXP{\nu_1}=e_1 \leq \frac{\beta_0}{\alpha_0} \tau$, and the result holds for $k=1$. Let us now assume that $e_k \leq \frac{\beta_{k-1}}{\alpha_{k-1}} \tau$. We need to show that $e_{k+1} \leq \frac{\beta_k}{\alpha_k} \tau$. By taking expectations from both sides of \eqref{conv lemma ineq} we have
	\begin{align*}
	e_{k+1} \leq(1-\alpha_k)e_k+\beta_k.
	\end{align*}
	By the inductive assumption and that $\alpha_k \leq 1$ we obtain
	\begin{align*}
	e_{k+1} \leq(1-\alpha_k)\frac{\beta_{k-1}}{\alpha_{k-1}} \tau+\beta_k.
	\end{align*}
	From the assumption $\frac{\beta_{k-1}}{\alpha_{k-1}} \leq \frac{\beta_k}{\alpha_k}(1+ \rho \alpha_k)$ and the preceding relation, we have
	\begin{align*}
	e_{k+1} \leq(1-\alpha_k)\frac{\beta_k}{\alpha_k}(1+ \rho \alpha_k) \tau+\beta_k.
	\end{align*}
	So we can write
	\begin{align*}
	e_{k+1} \leq\frac{\beta_k}{\alpha_k}\tau-(1-\rho)\beta_k\tau -\rho \alpha_k\beta_k\tau+\beta_k.
	\end{align*}
	Rearranging the terms, we obtain
	\begin{align*}
	e_{k+1} \leq \frac{\beta_k}{\alpha_k}\tau+\beta_k(-\tau(1-\rho)-\rho \tau \alpha_k+1) \leq  \frac{\beta_k}{\alpha_k}\tau+\beta_k\underbrace{(-\tau(1-\rho)+1)}_\text{\hbox{Term1}}.
	\end{align*}
	Since $\tau\geq \frac{1}{1-\rho}$ and $0<\rho<1
	$, Term1 is always nonpositive. So, $e_{k+1}\leq \frac{\beta_k}{\alpha_k} \tau$ and the proof is complete.
\end{proof}
\subsection{\textbf{Proof of Proposition \ref{prop condition for sequences}}} \label{proof of prop condition for sequences}
\begin{proof} \noindent (i) In the following, we show that conditions of Assumption \ref{assumptions for conv} are satisfied.
	\begin{itemize}
		\item[(a)] Note that $a,b>0$ and $\gamma_0\lambda_0 \leq \frac{L_\omega}{\mu_h}$ are sufficient for the sequences to be non-increasing and for Assumption \ref{assumptions for conv}(a) to be satisfied.
		\item[(b)] From the definition of $\gamma_k$ and $\lambda_k$ and that $a+b<1$, we have
		\begin{align*}
		\sum_{k=0}^{\infty}\gamma_k\lambda_k=\sum_{k=0}^{\infty}\frac{\gamma_0}{(k+1)^a}\frac{\lambda_0}{(k+1)^b}=\sum_{k=0}^{\infty}\frac{\gamma_0\lambda_0}{(k+1)^{a+b}}=\infty.
		\end{align*}
		\item[(c)] We have
		\begin{align*}
		\sum_{k=0}^{\infty}\frac{1}{\gamma_k\lambda_k}\left(\frac{\lambda_{k-1}}{\lambda_k}-1\right)^2=\sum_{k=0}^{\infty}\frac{(k+1)^{a+b}}{\gamma_0\lambda_0} \left(\underbrace{\left(1+\frac{1}{k}\right)^b-1}_\text{\hbox{Term1}}\right)^2.
		\end{align*}
		By Taylor series, we can write Term1$=(1+b/k+ b(b-1)/2k^2 + b(b-1)(b-2)/6k^3 + \cdots)-1={\cal
			O}\left(k^{-1}\right)$. So the preceding equality will be
		\begin{align*}
		\sum_{k=0}^{\infty}\frac{1}{\gamma_k\lambda_k}\left(\frac{\lambda_{k-1}}{\lambda_k}-1\right)^2=\sum_{k=0}^{\infty}\frac{(k+1)^{a+b}}{\gamma_0\lambda_0}{\cal
			O}\left(k^{-2}\right)=\sum_{k=0}^{\infty} {\cal O}\left(k^{-(2-(a+b))}\right).
		\end{align*}
		Since $a+b<1$, we have $2-(a+b)>1$. Therefore, the preceding series is summable implying that Assumption \ref{assumptions for conv}(c) is satisfied.
		\item[(d)] We have $\gamma_k^2 = \gamma_0^2/(k+1)^{2a} ={\cal O}\left(k^{-2a}\right)$. Since $a>0.5$, $\gamma_k^2$ is summable and Assumption \ref{assumptions for conv}(d) is met. 
		\item[(e)] In a similar fashion to part (c), we have
		\begin{align*}
		\lim_{k\to \infty} \frac{1}{\gamma_k^2\lambda_k^2}\left(\frac{\lambda_{k-1}}{\lambda_k}-1\right)^2=\lim_{k\to \infty} \frac{(k+1)^{2a+2b}}{\gamma_0\lambda_0}{\cal O}\left(k^{-2}\right).
		\end{align*}
		Since $a+b<1$, this limit goes to zero which implies that Assumption \ref{assumptions for conv}(e) is satisfied. 
		\item[(f)] The last condition in Assumption \ref{assumptions for conv} holds due to $a>b$.
	\end{itemize}	
	
	\noindent (ii) Next, we verify conditions of Assumption \ref{assumptions for general conv}.
	\begin{itemize}
		\item[(a)] The proof for this condition is identical to the proof given for Assumption \ref{assumptions for conv}.
		\item[(b)] By the analysis from part (1)(c) we have
		\begin{align*}
		\frac{1}{\gamma_k^3\lambda_k}\left(\frac{\lambda_{k-1}}{\lambda_k}-1\right)^2=\frac{(k+1)^{3a+b}}{\gamma_0\lambda_0}{\cal O}\left(k^{-2}\right)={\cal O}\left(k^{3a+b-2}\right).
		\end{align*}
		Note that since $3a+b<2$, the preceding term is bounded above by a constant. Therefore, there are constants $B_1$ and $k_1$ for which Assumption \ref{assumptions for general conv}(b) is satisfied.
		\item[(c)] We need to show that there are $0<\rho<1$ and $k_2$ such that the following holds:
		\begin{align*}
		\underbrace{\frac{\gamma_{k-1}\lambda_k}{\lambda_{k-1}\gamma_k} -1}_\text{\hbox{Term1}} \leq \rho\frac{\mu_h}{2L_\omega}\gamma_k\lambda_k \qquad \hbox{for all } k \geq k_2.
		\end{align*}
		Replacing $\gamma_k$ and $\lambda_k$ in Term1, we have Term1$=\left(1+1/k\right)^{a-b} -1 = {\cal O}\left(k^{-1}\right)$.
		By applying the same analysis as part (1)(c) and that $a>b$. Multiplying and dividing Term1 by $\rho\frac{\mu_h}{2L_\omega}\gamma_k\lambda_k$ in which $\rho$ can be any constant between 0 and 1, we obtain
		\begin{align*}
		\hbox{Term1}= \rho\frac{\mu_h}{2L_\omega}\gamma_k\lambda_k \frac{{\cal O}\left(k^{-1}\right)}{\rho\frac{\mu_h}{2L_\omega}\gamma_k\lambda_k}&=\rho\frac{\mu_h}{2L_\omega}\gamma_k\lambda_k \frac{{\cal O}\left(k^{-1}\right)}{\rho\frac{\mu_h}{2L_\omega}\gamma_0\lambda_0(k+1)^{-a-b}}\\&=\rho\frac{\mu_h}{2L_\omega}\gamma_k\lambda_k {\cal O}\left(k^{-1+a+b}\right).
		\end{align*}
		Since $a+b<1$, ${\cal O}\left(k^{-1+a+b}\right)$ converges to zero. So, there exists an integer $k_2$ such that for any  $k\geq k_2$ we have Term1 $\leq \rho\frac{\mu_h}{2L_\omega}\gamma_k\lambda_k$.  Thus, Assumption \ref{assumptions for general conv}(c) is satisfied. 
		\item[(d)] Condition (d) of Assumption \ref{assumptions for general conv} follows due to $a>b$.
	\end{itemize}
\end{proof} 
\end{document}